\documentclass[12pt]{article}
\usepackage[a4paper]{geometry}
\usepackage{amssymb,amsmath,amsthm}
\usepackage{graphicx,cite,xcolor}
\usepackage{longtable,pdflscape,booktabs,caption,multicol}
\usepackage[colorlinks=true,citecolor=black,linkcolor=black,urlcolor=blue]{hyperref}

\theoremstyle{plain}
\newtheorem{theorem}{Theorem}
\newtheorem{lemma}[theorem]{Lemma}
\newtheorem{corollary}[theorem]{Corollary}
\newtheorem{problem}[theorem]{Question}

\newtheorem{conjecture}[theorem]{Conjecture}

\theoremstyle{definition}

\theoremstyle{remark}

\newcommand{\arxiv}[1]{\href{https://arxiv.org/abs/#1}{\texttt{arXiv:#1}}}

\DeclareMathOperator{\Circ}{Circ}

\usepackage{listings}
\definecolor{mauve}{rgb}{0.58,0,0.82}
\lstdefinestyle{pitonche} {
    language = Python,
    basicstyle = footnotesizettfamily,
    showspaces = false,
    showstringspaces = false,
    breakautoindent = true,
    flexiblecolumns = true,
    keepspaces = true,
    stepnumber = 1,
    xleftmargin = 0pt
}
\usepackage{float}
\usepackage{tikz}
\usetikzlibrary{arrows}
\usetikzlibrary{arrows.meta}
\usetikzlibrary{chains}
\usetikzlibrary{positioning}
\usetikzlibrary{automata,positioning,calc}
\usetikzlibrary{decorations}
\usetikzlibrary{decorations.shapes}
\usetikzlibrary{decorations.markings}
\tikzset{
    edge/.style={-{Latex[scale=1.7]}},
    dedge/.style={{Latex[scale=1.7]}-{Latex[scale=1.7]}},
}

\title{Complete resolution of the circulant nut graph order--degree existence problem}

\author{Ivan Damnjanovi\'c\thanks{The author is supported by Diffine LLC.}\\
\small University of Ni\v s, Faculty of Electronic Engineering,\\[-0.4ex]
\small Aleksandra Medvedeva 14, 18106 Ni\v s, Serbia\\[-0.4ex]
\small\tt ivan.damnjanovic@elfak.ni.ac.rs\\
\small Diffine LLC\\[-0.4ex]
\small 3681 Villa Terrace, San Diego, CA 92104, USA\\[-0.4ex]
\small\tt ivan@diffine.com}

\begin{document}

\maketitle

\begin{abstract}
A circulant nut graph is a non-trivial simple graph such that its adjacency matrix is a circulant matrix whose null space is spanned by a single vector without zero elements. Regarding these graphs, the order--degree existence problem can be thought of as the mathematical problem of determining all the possible pairs $(n, d)$ for which there exists a $d$-regular circulant nut graph of order $n$. This problem was initiated by Ba\v{s}i\'c et al.\ and the first major results were obtained by Damnjanovi\'c and Stevanovi\'c, who proved that for each odd $t \ge 3$ such that $t\not\equiv_{10}1$ and $t\not\equiv_{18}15$, there exists a $4t$-regular circulant nut graph of order $n$ for each even $n \ge 4t + 4$. Afterwards, Damnjanovi\'c improved these results by showing that there necessarily exists a $4t$-regular circulant nut graph of order $n$ whenever $t$ is odd, $n$ is even, and $n \ge 4t + 4$ holds, or whenever $t$ is even, $n$ is such that $n \equiv_4 2$, and $n \ge 4t + 6$ holds. In this paper, we extend the aforementioned results by completely resolving the circulant nut graph order--degree existence problem. In other words, we fully determine all the possible pairs $(n, d)$ for which there exists a $d$-regular circulant nut graph of order $n$.

\bigskip\noindent
{\bf Mathematics Subject Classification:} 05C50, 11C08, 12D05, 13P05. \\
{\bf Keywords:} circulant graph, nut graph, graph spectrum, graph eigenvalue, cyclotomic polynomial.
\end{abstract}

\section{Introduction}\label{intro}

In this paper we will consider all graphs to be undirected, finite, simple and non-null. Thus, every graph will have at least one vertex and there shall be no loops or multiple edges. Also, for convenience, we will take that each graph of order $n$ has the vertex set $\{0, 1, 2, \ldots, n-1\}$.

A graph $G$ is considered to be a circulant graph if its adjacency matrix $A$ has the form
\[
A = \begin{bmatrix}
    a_0 & a_1 & a_2 & \cdots & a_{n-1}\\
    a_{n-1} & a_0 & a_1 & \cdots & a_{n-2}\\
    a_{n-2} & a_{n-1} & a_0 & \cdots & a_{n-3}\\
    \vdots & \vdots & \vdots & \ddots & \vdots\\
    a_1 & a_2 & a_3 & \dots & a_0
\end{bmatrix}.
\]
Here, we clearly have $a_0 = 0$, as well as $a_j = a_{n-j}$ for all $j = \overline{1, n-1}$. A concise way of describing a circulant graph is by taking into consideration the set of all the values $1 \le j \le \frac{n}{2}$ such that $a_j = a_{n-j} = 1$. We shall refer to this set as the generator set of a circulant graph and we will use $\mathrm{Circ}(n, S)$ to denote the circulant graph of order $n$ whose generator set is $S$.

A nut graph is a non-trivial graph whose adjacency matrix has nullity one and is such that its non-zero null space vectors have no zero elements, as first described by Sciriha in \cite{Sciriha}. Bearing this in mind, a circulant nut graph is simply a nut graph whose adjacency matrix additionally represents a circulant matrix. Ba\v{s}i\'c et al.\ \cite{Basic} initiated the study of these graphs by providing several results, alongside the following conjecture.

\begin{conjecture}[Ba\v{s}i\'c et al.\ \cite{Basic}]\label{basic}
For every $d$, where $d \equiv_4 0$, and for every even $n,\, n \ge d+4$, there exists a circulant nut graph $\mathrm{Circ}(n, \{s_1, s_2, s_3, \ldots, s_{\frac{d}{2}}\})$ of degree $d$.
\end{conjecture}

Damnjanovi\'c and Stevanovi\'c \cite[Lemma 18]{Damnjanovic_1} quickly disproved this conjecture by showing that for each $d$ such that $8 \mid d$, a $d$-regular circulant nut graph cannot have an order that is below $d+6$. However, Conjecture \ref{basic} did indirectly bring up an interesting question: \textit{``What are all the pairs $(n, d)$ for which there exists a $d$-regular circulant nut graph of order $n$?''}. Henceforth, we shall refer to this mathematical problem as the circulant nut graph order--degree existence problem, and we will use $\mathcal{N}_d$ to denote the set of all the orders that a $d$-regular circulant nut graph can have, for each $d \in \mathbb{N}_0$.

\newpage
Regarding the aforementioned problem, there are several basic facts that can quickly be noticed, as demonstrated by Damnjanovi\'c and Stevanovi\'c \cite{Damnjanovic_1}. First of all, it is easy to show that every $d$-regular circulant nut graph of order $n$ must satisfy $4 \mid d$ and $2 \mid n$. Moreover, for any odd $t \in \mathbb{N}$, a $4t$-regular circulant nut graph cannot have an order below $4t + 4$, while for any even $t \in \mathbb{N}$, such a graph cannot have an order smaller than $4t + 6$, as already discussed.

Furthermore, Damnjanovi\'c and Stevanovi\'c \cite{Damnjanovic_1} have managed to construct a $4t$-regular circulant nut graph of order $n$ for each even $n \ge 4t + 4$, provided $t$ is odd, $t \not\equiv_{10} 1$ and $t \not\equiv_{18} 15$. This result is disclosed in the following theorem.

\begin{theorem}[Damnjanovi\'c and Stevanovi\'c \cite{Damnjanovic_1}]\label{damnjanovic_theorem_1}
For each odd $t \ge 3$ such that $t\not\equiv_{10} 1$ and $t\not\equiv_{18} 15$,
the circulant graph $\Circ(n, \{1, 2, 3, \ldots, 2t+1\} \setminus \{ t \})$ is a nut graph for each even $n \ge 4t+4$.
\end{theorem}

Thus, Theorem \ref{damnjanovic_theorem_1} fully determines $\mathcal{N}_{4t}$ for infinitely many odd values of $t$. On top of that, Damnjanovi\'c and Stevanovi\'c \cite[Proposition 19]{Damnjanovic_1} have also found the set
\begin{equation}\label{d_is_8}
    \mathcal{N}_8 = \{14\} \cup \{ n \in \mathbb{N} \colon 2 \mid n \land n \ge 18 \}.
\end{equation}
In this scenario, it is interesting to notice that a surprising ``irregularity'' exists due to the absence of an $8$-regular circulant nut graph of order $16$. 

Afterwards, Damnjanovi\'c \cite{Damnjanovic_2} succeeded in improving the previously disclosed results by finding the set $\mathcal{N}_{4t}$ for each odd $t \in \mathbb{N}$:
\[
    \mathcal{N}_{4t} = \{ n \in \mathbb{N} \colon 2 \mid n \land n \ge 4t + 4 \} \qquad (\forall t \in \mathbb{N},\, 2 \nmid t).
\]
This result is an immediate corollary of the next two theorems.

\begin{theorem}[Damnjanovi\'c \cite{Damnjanovic_2}]\label{damnjanovic_theorem_2}
    For each odd $t \in \mathbb{N}$ and $n \ge 4t + 4$ such that $4 \mid n$, the circulant graph
    \[
        \Circ\left(n, \{1, 2, \ldots, t-1\} \cup \left\{\frac{n}{4}, \frac{n}{4} + 1 \right\} \cup \left\{\frac{n}{2} - (t-1), \ldots, \frac{n}{2} - 2, \frac{n}{2} - 1 \right\}\right)
    \]
    must be a $4t$-regular nut graph of order $n$.
\end{theorem}

\begin{theorem}[Damnjanovi\'c \cite{Damnjanovic_2}]\label{damnjanovic_theorem_3}
    For each $t \in \mathbb{N}$ and $n \ge 4t + 6$ such that $n \equiv_4 2$, the circulant graph
    \[
        \Circ\left(n, \{1, 2, \ldots, t-1\} \cup \left\{\frac{n+2}{4}, \frac{n+6}{4} \right\} \cup \left\{\frac{n}{2} - (t-1), \ldots, \frac{n}{2} - 2, \frac{n}{2}-1 \right\}\right)
    \]
    must be a $4t$-regular nut graph of order $n$.
\end{theorem}

In this paper, we fully resolve the circulant nut graph order--degree existence problem by finding $\mathcal{N}_d$ for each $d \in \mathbb{N}_0$. The main result is given in the following theorem.
\begin{theorem}[Circulant nut graph order--degree existence theorem]\label{main_theorem}
    For each \linebreak $d \in \mathbb{N}_0$, the set $\mathcal{N}_d$ can be determined via the following expression:
    \begin{equation}\label{main_theorem_formula}
        \mathcal{N}_d = \begin{cases}
            \varnothing,& d = 0 \lor 4 \nmid d,\\
            \{ n \in \mathbb{N} \colon 2 \mid n \land n \ge d + 4 \},& d \equiv_8 4,\\
            \{14\} \cup \{ n \in \mathbb{N} \colon 2 \mid n \land n \ge 18 \},& d = 8,\\
            \{ n \in \mathbb{N} \colon 2 \mid n \land n \ge d + 6 \},& 8 \mid d \land d \ge 16 .
        \end{cases}
    \end{equation}
\end{theorem}

The result given in the case $d = 0 \lor 4 \nmid d$ of Eq.\ (\ref{main_theorem_formula}) is straightforward to see, while the expression corresponding to the case $d = 8$ follows directly from Eq.~(\ref{d_is_8}). Given the fact that the case $d \equiv_8 4$ represents an immediate corollary of Theorems~\ref{damnjanovic_theorem_2} and \ref{damnjanovic_theorem_3}, as we have already mentioned, the only remaining case left to be proved is when $8 \mid d \land d \ge 16$. However, Theorem \ref{damnjanovic_theorem_3} tells us that for each such $d$, there does exist a circulant nut graph of every order $n$ such that $n \equiv_4 2$ and $n \ge d + 6$. Thus, taking everything into consideration, in order to complete the proof of Theorem \ref{main_theorem}, it only remains to be shown that for each even $t \ge 4$, there must exist a $4t$-regular circulant nut graph of each order $n$ such that $4 \mid n$ and $n \ge 4t + 8$. This is precisely the task that the remainder of the paper will solve.

The structure of the paper shall be organized in the following manner. After Section~\ref{intro}, which is the introduction, Section \ref{preliminaries} will serve to preview certain theoretical facts regarding the circulant matrices, circulant nut graphs and cyclotomic polynomials which are required to successfully finalize the proof of Theorem \ref{main_theorem}. Afterwards, we shall use three separate constructions in order to show the existence of all the required circulant nut graphs. In Section \ref{section_3} we will construct a $4t$-regular circulant nut graph of order $4t+8$, for each even $t \ge 4$, thereby showing that such a graph necessarily exists. After that, Section \ref{section_4} will be used to show that, for any even $t \ge 4$, there exists a $4t$-regular circulant nut graph of order $n$ for each $n \ge 4t + 16$ such that $8 \mid n$. Subsequently, Section \ref{section_5} will demonstrate the existence of a $4t$-regular circulant nut graph of order $n$ for each $n \ge 4t + 12$ such that $n \equiv_8 4$, where $t \ge 4$ is an arbitrarily chosen even integer. Finally, Section~\ref{conclusion} shall provide a brief conclusion regarding all the obtained results and give two additional problems to be examined in the future.

\section{Preliminaries}\label{preliminaries}

It is known from elementary linear algebra theory (see, for example, \cite[Section~3.1]{Gray}) that the circulant matrix
\[
A = \begin{bmatrix}
    a_0 & a_1 & a_2 & \cdots & a_{n-1}\\
    a_{n-1} & a_0 & a_1 & \cdots & a_{n-2}\\
    a_{n-2} & a_{n-1} & a_0 & \cdots & a_{n-3}\\
    \vdots & \vdots & \vdots & \ddots & \vdots\\
    a_1 & a_2 & a_3 & \dots & a_0
\end{bmatrix}
\]
must have the eigenvalues
\[
    P(1), P(\omega), P(\omega^2), \ldots, P(\omega^{n-1}),
\]
where $\omega=e^{i \frac{2 \pi}{n}}$ is an $n$-th root of unity, and
\[
    P(x) = a_0 + a_1 x + a_2 x^2 + \cdots + a_{n-1} x^{n-1} .
\]

Starting from the aforementioned result, Damnjanovi\'c and Stevanovi\'c \cite{Damnjanovic_1} have managed to give the necessary and sufficient conditions for a circulant graph to be a nut graph in the form of the following lemma.

\begin{lemma}[Damnjanovi\'c and Stevanovi\'c \cite{Damnjanovic_1}]\label{damnjanovic_lemma_1}
    Let $G = \Circ(n, S)$ where $n \ge 2$. The graph $G$ is a nut graph if and only if all of the following conditions hold:
    \begin{itemize}
        \item $2 \mid n$;
        \item $S$ consists of $t$ odd and $t$ even integers from $\left\{1, 2, 3, \ldots, \frac{n}{2} - 1 \right\}$, for some $t \ge 1$;
        \item $P(\omega^j)\neq 0$ for each $j\in\left\{1, 2, 3, \ldots, \frac{n}{2}-1 \right\}$.
    \end{itemize}
\end{lemma}

Suppose that we are given an arbitrary circulant graph of even order $n$ whose generator set is non-empty and contains equally many odd and even integers, all of which are positive integers smaller than $\frac{n}{2}$. Taking into consideration Lemma~\ref{damnjanovic_lemma_1}, it becomes apparent that in order to show that such a graph is a nut graph, it is sufficient to prove that it satisfies the third condition given in the lemma. In other words, it is enough to demonstrate that, for this graph, the polynomial $P(x) \in \mathbb{Z}[x]$ has no $n$-th roots of unity among its roots, except potentially $-1$ or $1$.

Furthermore, it is clear that $\zeta^{n-j} = \dfrac{1}{\zeta^j}$ for each $j = \overline{1, n-1}$ and each $n$-th root of unity $\zeta \in \mathbb{C}$. Bearing this in mind, we quickly obtain that
\begin{equation}\label{polynomial_formula}
    P(\zeta) = \left( \zeta^{s_0} + \frac{1}{\zeta^{s_0}} \right) + \left( \zeta^{s_1} + \frac{1}{\zeta^{s_1}} \right) + \cdots + \left( \zeta^{s_{k-1}} + \frac{1}{\zeta^{s_{k-1}}} \right)
\end{equation}
for an arbitrary $n$-th root of unity $\zeta$ and circulant graph $G = \mathrm{Circ}(n, S)$, where $S = \{ s_0, s_1, s_2, \ldots, s_{k-1} \}$, provided all the generator set elements are lower than $\frac{n}{2}$. Sections \ref{section_3}, \ref{section_4} and \ref{section_5} will all heavily rely on Eq.\ (\ref{polynomial_formula}), as well as Lemma \ref{damnjanovic_lemma_1}, whilst proving that the soon-to-be constructed circulant graphs are indeed nut graphs.

Last but not least, it is crucial to point out that the cyclotomic polynomials shall play a key role in demonstrating whether or not certain polynomials of interest contain the given roots of unity among their roots. The cyclotomic polynomial $\Phi_b(x)$ can be defined for each $b \in \mathbb{N}$ via
\[
    \Phi_b(x) = \prod_{\xi} (x - \xi) ,
\]
where $\xi$ ranges over the primitive $b$-th roots of unity. It is known that these polynomials have integer coefficients and that they are all irreducible in $\mathbb{Q}[x]$ (see, for example, \cite{cyclotomic}). Hence, an arbitrary polynomial in $\mathbb{Q}[x]$ has a primitive $b$-th root of unity among its roots if and only if it is divisible by $\Phi_b(x)$.

While inspecting whether certain integer polynomials are divisible by cyclotomic polynomials, we will strongly rely on the following theorem on the divisibility of lacunary polynomials by cyclotomic polynomials.
\begin{theorem}[Filaseta and Schinzel \cite{Filaseta}]\label{filaseta}
Let $P(x) \in \mathbb{Z}[x]$ have $N$ nonzero terms and let $\Phi_b(x) \mid P(x)$.
Suppose that $p_1, p_2, \dots, p_k$ are distinct primes such that
\[
    \sum_{j=1}^k (p_j-2) > N-2 .
\]
Let $e_j$ be the largest exponent such that $p_j^{e_j} \mid b$. Then for at least one $j$, $1 \le j \le k$, we have that $\Phi_{b'}(x)\mid P(x)$, where $b' = \dfrac{b}{p_j^{e_j}}$.
\end{theorem}

\section{Construction for \texorpdfstring{$n = 4t + 8$}{n = 4t + 8}}\label{section_3}

In this section, we will demonstrate that for each even $t \ge 4$ there does exist a $4t$-regular circulant nut graph of order $4t + 8$. In order to achieve this, we shall provide a concrete example of such a graph, for each even $t \ge 4$, and then prove that the given graph is indeed a circulant nut graph. While constructing these graphs, we will rely on two different construction patterns. One pattern will be used for the scenario when $4 \mid t$, while the second will give us our desired result provided $t \equiv_4 2$. In the rest of the section we present the two according lemmas.

\newpage
\begin{lemma}\label{4t+8}
    For each $t \ge 4$ such that $4 \mid t$, the circulant graph 
    \[
        \Circ(4t+8, \{1, 2, 3, \ldots, 2t + 3\} \setminus \{ t+1, t+3, t+4 \})
    \]
    must be a $4t$-regular circulant nut graph of order $4t+8$.
\end{lemma}
\begin{proof}
    Let $n = 4t + 8$. First of all, we know that $t+1$ and $t+3$ are odd, while $t+4$ is even, which directly tells us that the given circulant graph does have a non-empty generator set that contains equally many odd and even integers, all of which are positive, but smaller than $\frac{n}{2}$. Thus, by virtue of Lemma \ref{damnjanovic_lemma_1}, in order to prove the given lemma, it is sufficient to show that the polynomial $P(x)$ has no $n$-th roots of unity among its roots, except potentially $1$ or $-1$.

    Let $\zeta \in \mathbb{C}$ be an arbitrary $n$-th root of unity that is different from both $1$ and $-1$. By implementing Eq.\ (\ref{polynomial_formula}), we swiftly obtain
    \[
        P(\zeta) = \sum_{j = 1}^{2t+3}\left( \zeta^j + \zeta^{-j} \right) - \left( \zeta^{t+1} + \zeta^{-t-1} \right) - \left( \zeta^{t+3} + \zeta^{-t-3} \right) - \left( \zeta^{t+4} + \zeta^{-t-4} \right) .
    \]
    However, since $\zeta \neq 1$, we know that
    \begin{alignat*}{2}
        && \sum_{j = 0}^{n-1} \zeta^j &= 0\\
        \implies \quad && \zeta^{-2t-3} \sum_{j = 0}^{4t+7} \zeta^j &= 0\\
        \implies \quad && \sum_{j = -2t-3}^{2t+4} \zeta^j &= 0\\
        \implies \quad && \zeta^{2t+4} + 1 + \sum_{j = 1}^{2t+3} \left( \zeta^j + \zeta^{-j}\right) &= 0\\
        \implies \quad && \sum_{j = 1}^{2t+3} \left( \zeta^j + \zeta^{-j}\right) &= -1 - \zeta^{2t+4} .
    \end{alignat*}
    Thus, the condition $P(\zeta) = 0$ quickly becomes equivalent to
    \begin{alignat}{2}
        \nonumber && P(\zeta) &= 0\\
        \nonumber\iff \quad && -1 - \zeta^{2t+4} - \zeta^{t+1} - \zeta^{-t-1} - \zeta^{t+3} - \zeta^{-t-3} - \zeta^{t+4} - \zeta^{-t-4} &= 0\\
        \nonumber\iff \quad && -\zeta^{t+4}(-1 - \zeta^{2t+4} - \zeta^{t+1} - \zeta^{-t-1} - \zeta^{t+3} - \zeta^{-t-3} - \zeta^{t+4} - \zeta^{-t-4}) &= 0\\
        \label{aux_1}\iff \quad && \zeta^{3t+8} + \zeta^{2t+8} + \zeta^{2t+7} + \zeta^{2t+5} + \zeta^{t+4} + \zeta^3 + \zeta + 1 &= 0 .
    \end{alignat}
    We will finish the proof of the lemma by dividing the problem into two cases depending on the value of $\zeta^{\frac{n}{2}}$.

    \bigskip\noindent
    \emph{Case $\zeta^{\frac{n}{2}} = -1$}.\quad
    In this case, we have $\zeta^{2t+4} = -1$, hence $\zeta^{3t+8} = -\zeta^{t+4}$, which means that Eq.\ (\ref{aux_1}) leads us to
    \begin{alignat*}{2}
        && P(\zeta) &= 0\\
        \iff \quad &&\zeta^{2t+8} + \zeta^{2t+7} + \zeta^{2t+5} + \zeta^3 + \zeta + 1 &= 0\\
        \iff \quad && -\zeta^4 - \zeta^3 - \zeta + \zeta^3 + \zeta + 1 &= 0\\
        \iff \quad && 1 - \zeta^4 &= 0\\
        \iff \quad && \zeta^4 &= 1 .
    \end{alignat*}
    However, $\zeta^4 = 1$ cannot possibly hold. Moreover, $\zeta \neq 1, -1$ by definition, while $i$ and $-i$ do not satisfy the conditions $i^{\frac{n}{2}} = -1$ and $(-i)^{\frac{n}{2}} = -1$ due to the fact that $4 \mid 2t + 4$. Thus, $P(\zeta) = 0$ does not hold for any $n$-th root of unity $\zeta$ that is different from both $1$ and $-1$ and such that $\zeta^{\frac{n}{2}} = -1$.

    \bigskip\noindent
    \emph{Case $\zeta^{\frac{n}{2}} = 1$}.\quad
    In this scenario, we immediately see that $\zeta^{3t+8} = \zeta^{t+4}$, which further helps us obtain from Eq.\ (\ref{aux_1})
    \begin{alignat}{2}
        \nonumber && P(\zeta) &= 0\\
        \nonumber\iff \quad &&\zeta^{2t+8} + \zeta^{2t+7} + \zeta^{2t+5} + 2\zeta^{t+4} + \zeta^3 + \zeta + 1 &= 0\\
        \nonumber\iff \quad && \zeta^4 + \zeta^3 + \zeta + 2\zeta^{t+4} + \zeta^3 + \zeta + 1 &= 0\\
        \label{aux_2} \iff \quad && 2\zeta^{t+4} + \zeta^4 + 2\zeta^3 + 2\zeta + 1&= 0 .
    \end{alignat}
    We now divide the problem into two subcases depending on the value of $\zeta^{\frac{n}{4}}$.

    \medskip\noindent
    \emph{Subcase $\zeta^{\frac{n}{4}} = -1$}.\quad
    Here, it is clear that $\zeta^{t+4} = -\zeta^2$, which means that Eq.\ (\ref{aux_2}) directly transforms to
    \[
        P(\zeta) = 0 \quad \iff \quad \zeta^4 + 2\zeta^3 - 2\zeta^2 + 2\zeta + 1 = 0 .
    \]
    However, the polynomial $x^4 + 2x^3 - 2x^2 + 2x + 1 \in \mathbb{Q}[x]$ has no roots of unity among its roots, as demonstrated in Appendix \ref{problematic_roots}. This means that $P(\zeta) = 0$ cannot possibly hold for any $\zeta$ that is an $n$-th root of unity, as desired.

    \medskip\noindent
    \emph{Subcase $\zeta^{\frac{n}{4}} = 1$}.\quad
    In this subcase, we obtain $\zeta^{t+4} = \zeta^2$. Thus, Eq.\ (\ref{aux_2}) gives us
    \begin{alignat*}{2}
        && P(\zeta) &= 0\\
        \iff \quad && \zeta^4 + 2\zeta^3 + 2\zeta^2 + 2\zeta + 1 &= 0\\
        \iff \quad && (\zeta^2 + 1)(\zeta+1)^2 &= 0\\
        \iff \quad && \zeta^2 + 1 &= 0 .
    \end{alignat*}
    Now, by taking into consideration that $i^{\frac{n}{4}} = (-i)^{\frac{n}{4}} = -1$ due to the fact that $\frac{n}{4} = t + 2 \equiv_4 2$, we clearly see that for any $n$-th root of unity $\zeta \in \mathbb{C}$ different from $1$ and $-1$ and such that $\zeta^\frac{n}{4} = 1$, the equality $\zeta^2 + 1 = 0$ truly cannot hold. Hence, we reach $P(\zeta) \neq 0$ once again.
\end{proof}

\begin{lemma}
    For each $t \ge 6$ such that $t \equiv_4 2$, the circulant graph 
    \[
        \Circ(4t+8, \{1, 2, 3 \ldots, 2t + 3\} \setminus \{ t-2, t+1, t+3 \})
    \]
    must be a $4t$-regular circulant nut graph of order $4t+8$.
\end{lemma}
\begin{proof}
    Let $n = 4t + 8$. It is clear that $t-2$ is even, while $t+1$ and $t+3$ are odd, which implies that the given circulant graph has a non-empty generator set that contains equally many odd and even integers, all of which are positive and lower than $\frac{n}{2}$. By relying on Lemma \ref{damnjanovic_lemma_1}, we know that in order to finalize the proof of the lemma, it is enough to demonstrate that the polynomial $P(x)$ has no $n$-th roots of unity among its roots, except potentially $1$ or $-1$.

    We will use a very similar strategy to complete the proof as it was done in Lemma \ref{4t+8}. Let $\zeta \in \mathbb{C}$ be an arbitrary $n$-th root of unity such that $\zeta \neq 1, -1$. By using Eq.\ (\ref{polynomial_formula}), we immediately get
    \[
        P(\zeta) = \sum_{j = 1}^{2t+3}\left( \zeta^j + \zeta^{-j} \right) - \left( \zeta^{t-2} + \zeta^{-t+2} \right) - \left( \zeta^{t+1} + \zeta^{-t-1} \right) - \left( \zeta^{t+3} + \zeta^{-t-3} \right) .
    \]
    Now, we can use the same equality $\displaystyle\sum_{j = 0}^{2t+3} \left( \zeta^j + \zeta^{-j}\right) = -1 - \zeta^{2t+4}$ that was proved in Lemma \ref{4t+8} in order to conclude that
    \begin{alignat}{2}
        \nonumber && P(\zeta) &= 0\\
        \nonumber\iff \quad && -1 - \zeta^{2t+4} - \zeta^{t-2} - \zeta^{-t+2} - \zeta^{t+1} - \zeta^{-t-1} - \zeta^{t+3} - \zeta^{-t-3} &= 0\\
        \nonumber\iff \quad && -\zeta^{t+3}(-1 - \zeta^{2t+4} - \zeta^{t-2} - \zeta^{-t+2} - \zeta^{t+1} - \zeta^{-t-1} - \zeta^{t+3} - \zeta^{-t-3}) &= 0\\
        \label{aux_3}\iff \quad && \zeta^{3t+7} + \zeta^{2t+6} + \zeta^{2t+4} + \zeta^{2t+1} + \zeta^{t+3} + \zeta^5 + \zeta^2 + 1 &= 0 .
    \end{alignat}
    We shall finish the proof by dividing the problem into two cases depending on the value of $\zeta^{\frac{n}{2}}$.

    \newpage
    \bigskip\noindent
    \emph{Case $\zeta^{\frac{n}{2}} = -1$}.\quad
    Here, we see that $\zeta^{2t+4} = -1$, hence $\zeta^{3t+7} = -\zeta^{t+3}$. On behalf of Eq.\ (\ref{aux_3}), $P(\zeta) = 0$ becomes further equivalent to
    \begin{alignat*}{2}
        && P(\zeta) &= 0\\
        \iff \quad && \zeta^{2t+6} + \zeta^{2t+4} + \zeta^{2t+1} + \zeta^5 + \zeta^2 + 1 &= 0\\
        \iff \quad && -\zeta^2 - 1 - \frac{1}{\zeta^3} + \zeta^5 + \zeta^2 + 1 &= 0\\
        \iff \quad && \zeta^5 - \frac{1}{\zeta^3} &= 0\\
        \iff \quad && \zeta^8 &= 1 .
    \end{alignat*}
    However, $\frac{n}{2} = 2t + 4$, where $t \equiv_4 2$, which means that $8 \mid \frac{n}{2}$. This implies that whenever some eighth root of unity is raised to the power of $\frac{n}{2}$, it yields $1$, not $-1$. Hence, the equality $\zeta^8 = 1$ cannot possibly hold for any $n$-th root of unity $\zeta \in \mathbb{C}$ such that $\zeta^{\frac{n}{2}} = -1$. Thus, we obtain $P(\zeta) \neq 0$, as desired.

    \bigskip\noindent
    \emph{Case $\zeta^{\frac{n}{2}} = 1$}.\quad
    In this case, it is clear that $\zeta^{3t+7} = \zeta^{t+3}$, which allows us to implement Eq.\ (\ref{aux_3}) in order to reach
    \begin{alignat}{2}
        \nonumber && P(\zeta) &= 0\\
        \nonumber\iff \quad &&\zeta^{2t+6} + \zeta^{2t+4} + \zeta^{2t+1} + 2\zeta^{t+3} + \zeta^5 + \zeta^2 + 1 &= 0\\
        \nonumber\iff \quad && \zeta^2 + 1 + \frac{1}{\zeta^3} + 2\zeta^{t+3} + \zeta^5 + \zeta^2 + 1 &= 0\\
        \nonumber\iff \quad && \zeta^3 \left( 2\zeta^{t+3} + \zeta^5 + 2\zeta^2 + 2 + \frac{1}{\zeta^3} \right) &= 0\\
        \label{aux_4} \iff \quad && 2\zeta^{t+6} + \zeta^8 + 2\zeta^5 + 2\zeta^3 + 1 &= 0 .
    \end{alignat}
    We now divide the problem into two subcases depending on the value of $\zeta^{\frac{n}{4}}$.

    \medskip\noindent
    \emph{Subcase $\zeta^{\frac{n}{4}} = -1$}.\quad
    In this subcase, we know that $\zeta^{t+6} = -\zeta^4$, hence Eq.\ (\ref{aux_4}) quickly implies
    \begin{alignat*}{2}
        && P(\zeta) &= 0\\
        \iff \quad && \zeta^8 + 2\zeta^5 - 2 \zeta^4 + 2\zeta^3 + 1 &= 0 \\
        \iff \quad && (\zeta^2 + 1)(\zeta^6 - \zeta^4 + 2\zeta^3 - \zeta^2 + 1) &= 0 .
    \end{alignat*}
    Furthermore, we have $i^\frac{n}{4} = (-i)^\frac{n}{4} = 1$ due to the fact that $\frac{n}{4} = t + 2 \equiv_4 0$, which means that $\zeta^2 + 1 \neq 0$. This leads us to
    \[
        P(\zeta) = 0 \quad \iff \quad \zeta^6 - \zeta^4 + 2\zeta^3 - \zeta^2 + 1 = 0 .
    \]
    However, the polynomial $x^6 - x^4 + 2x^3 - x^2 + 1 \in \mathbb{Q}[x]$ has no roots of unity among its roots, as shown in Appendix \ref{problematic_roots}. This implies that $P(\zeta) \neq 0$ for any $n$-th root of unity $\zeta$ such that $\zeta \neq 1, -1$ and $\zeta^\frac{n}{4} = -1$.

    \medskip\noindent
    \emph{Subcase $\zeta^{\frac{n}{4}} = 1$}.\quad
    Here, we get $\zeta^{t+6} = \zeta^4$, which enables us to use Eq.\ (\ref{aux_4}) to swiftly obtain
    \begin{alignat*}{2}
        && P(\zeta) &= 0\\
        \iff \quad && \zeta^8 + 2\zeta^5 + 2 \zeta^4 + 2\zeta^3 + 1 &= 0 \\
        \iff \quad && (\zeta + 1)^2 (\zeta^6 - 2\zeta^5 + 3 \zeta^4 - 2\zeta^3 + 3\zeta^2 - 2 \zeta + 1) &= 0 \\
        \iff \quad && \zeta^6 - 2\zeta^5 + 3 \zeta^4 - 2\zeta^3 + 3\zeta^2 - 2 \zeta + 1 &= 0 .
    \end{alignat*}
    The polynomial $x^6 - 2x^5 + 3x^4 - 2x^3 + 3x^2 - 2x + 1 \in \mathbb{Q}[x]$ has no roots of unity among its roots, as demonstrated in Appendix \ref{problematic_roots}. This clearly shows that $P(\zeta) = 0$ cannot hold, as desired.
\end{proof}

\section{Construction for \texorpdfstring{$8 \mid n \land n \ge 4t + 16$}{8 | n and n >= 4t + 16}}\label{section_4}

In this section we will give a constructive proof of the existence of a $4t$-regular circulant nut graph of any order $n \in \mathbb{N}$ such that $n \ge 4t + 16$ and $8 \mid n$, for any even $t \ge 4$. In order to achieve this, we will prove the following theorem.

\begin{theorem}\label{main_th_1}
    For any even $t \ge 4$ and any $n \ge 4t + 16$ such that $8 \mid n$, the circulant graph $\mathrm{Circ}(n, S'_{t, n})$ where
    \begin{align*}
        S'_{t, n} = \{1, 2, \ldots, t-3 \} &\cup \{ t-1, t \} \cup \left\{ \frac{n}{4}, \frac{n}{4} + 2 \right\}\\
        &\cup \left\{ \frac{n}{2} - t, \frac{n}{2} - (t-1) \right\} \cup \left\{\frac{n}{2} - (t-3), \ldots, \frac{n}{2} - 2, \frac{n}{2} -1 \right\}
    \end{align*}
    must be a $4t$-regular circulant nut graph of order $n$.
\end{theorem}

For starters, it is clear that the set $S'_{t, n}$ is well defined, given the fact that $t < \frac{n}{4}$ and $\frac{n}{4} + 2 < \frac{n}{2} - t$ for each even $t \ge 4$ and each $n \ge 4t + 16$ such that $8 \mid n$. Moreover, it is not difficult to see that this set necessarily contains equally many odd and even integers, all of which are positive and smaller than $\frac{n}{2}$. By virtue of Lemma \ref{damnjanovic_lemma_1}, in order to prove Theorem \ref{main_th_1}, it is sufficient to show that $P(x)$ has no $n$-th roots of unity among its roots, except potentially $1$ or $-1$.

\newpage
The proof of Theorem \ref{main_th_1} will be carried out in a fashion that is very similar to the strategy used by Damnjanovi\'c \cite{Damnjanovic_2}. Thus, we will rely on a few auxiliary lemmas which will be used in order to finalize the proof in a more concise manner. We start off by defining the following two polynomials
\begin{align*}
    Q_t(x) = 2x^{2t+1} - 2x^{2t-1} + 2x^{2t-2} &+ x^{t+3} - x^{t+2} + x^{t-1} - x^{t-2} - 2x^3 + 2x^2 - 2 ,\\
    R_t(x) = 2x^{2t+1} - 2x^{2t-1} + 2x^{2t-2} &- x^{t+3} + x^{t+2} - 4x^{t+1}\\
    &+ 4x^t - x^{t-1} + x^{t-2} - 2x^3 + 2x^2 - 2 ,
\end{align*}
for each even $t \ge 6$. Now, since it is clear that $3 < t-2$ and $t+3 < 2t-2$ hold for any even $t \ge 6$, we see that $Q_t(x)$ must have exactly $10$ non-zero terms, while $R_t(x)$ surely has exactly $12$ non-zero terms. Let $L'_t$ and $L''_t$ be the sets containing the powers of these terms, respectively, i.e.\
\begin{align*}
    L'_t &= \{0, 2, 3, t-2, t-1, t+2, t+3, 2t-2, 2t-1, 2t+1\},\\
    L''_t &= \{0, 2, 3, t-2, t-1, t, t+1, t+2, t+3, 2t-2, 2t-1, 2t+1\},
\end{align*}
for each even $t \ge 6$. In the next lemma we will show one valuable property regarding these two sets.
\begin{lemma}\label{unique_remainder_1}
    For each even $t \ge 6$ and each $\beta \in \mathbb{N},\, \beta \ge 10$, $L'_t$ must contain an element whose remainder modulo $\beta$ is unique within the set. Also, for each even $t \ge 6$ and each $\beta \in \mathbb{N},\, \beta \ge 7$ such that $ \beta \nmid t$, $L''_t$ necessarily contains an element whose remainder modulo $\beta$ is unique within the set.
\end{lemma}
\begin{proof}
    It is clear that, for any $\beta \ge 7$, the elements $t-2, t-1, t, t+1, t+2, t+3$ must all have mutually distinct remainders modulo $\beta$. If the element $t-1$ were to have a distinct remainder modulo $\beta$ from all the remainders of the elements $0, 2, 3, 2t-2, 2t-1, 2t+1$, then this value would represent an element of the set $L'_t$, as well as the set $L''_t$, which has a unique remainder modulo $\beta$ inside the said set. The lemma statement would swiftly follow from here. Now, suppose otherwise, i.e.\ that the value $t-1$ does have the same remainder modulo $\beta$ as some number from the set $\{0, 2, 3, 2t-2, 2t-1, 2t+1\}$. We will finish the proof off by showing that the lemma statement holds in this scenario as well. For convenience, we will divide the problem into six corresponding cases.

    \begin{table}[h!t]
    {\footnotesize
    \begin{center}
    \begin{tabular}{rccccc}
    \toprule & $t \equiv_\beta -2$ & $t \equiv_\beta 0$ & $t \equiv_\beta 1$ & $t \equiv_\beta 3$ & $t \equiv_\beta 4$ \\
    \midrule
    $0 \equiv_\beta$ & $0$ & $0$ & $0$ & $0$ & $0$\\
    $2 \equiv_\beta$ & $2$ & $2$ & $2$ & $2$ & $2$\\
    $3 \equiv_\beta$ & $3$ & $3$ & $3$ & $3$ & $3$\\
    $t-2 \equiv_\beta$ & $-4$ & $-2$ & $-1$ & $1$ & $2$\\
    $t-1 \equiv_\beta$ & $-3$ & $-1$ & $0$ & $2$ & $3$\\
    $t \equiv_\beta$ & $-2$ & $0$ & $1$ & $3$ & $4$\\
    $t+1 \equiv_\beta$ & $-1$ & $1$ & $2$ & $4$ & $5$\\
    $t+2 \equiv_\beta$ & $0$ & $2$ & $3$ & $5$ & $6$\\
    $t+3 \equiv_\beta$ & $1$ & $3$ & $4$ & $6$ & $7$\\
    $2t-2 \equiv_\beta$ & $-6$ & $-2$ & $0$ & $4$ & $6$\\
    $2t-1 \equiv_\beta$ & $-5$ & $-1$ & $1$ & $5$ & $7$\\
    $2t+1 \equiv_\beta$ & $-3$ & $1$ & $3$ & $7$ & $9$\\
    \bottomrule
    \end{tabular}
    \end{center}
    \caption{The elements of the sets $L'_t$ and $L''_t$ modulo $\beta$, for certain values of $t \bmod \beta$.}
    \label{some_remainders}
    }
    \end{table}

    \bigskip\noindent
    \emph{Case $t-1 \equiv_\beta 0$}.\quad
    In this case we obtain $t \equiv_\beta 1$. From Table \ref{some_remainders} it is now clear that the element $t-2$ must have a unique remainder modulo $\beta$ in both $L'_t$ and $L''_t$.

    \bigskip\noindent
    \emph{Case $t-1 \equiv_\beta 2$}.\quad
    In this case we get $t \equiv_\beta 3$. Once again, Table \ref{some_remainders} tells us that the element $t-2$ must have a unique remainder modulo $\beta$ in both $L'_t$ and $L''_t$.

    \bigskip\noindent
    \emph{Case $t-1 \equiv_\beta 3$}.\quad
    Here, we conclude that $t \equiv_\beta 4$. According to Table \ref{some_remainders}, we see that the element $t$ necessarily has a unique remainder modulo $\beta$ in $L''_t$, whenever $\beta \ge 7$. On the other hand, if $\beta \ge 10$, then the element $0$ certainly has a unique remainder modulo $\beta$ within the set $L'_t$.

    \bigskip\noindent
    \emph{Case $t-1 \equiv_\beta 2t-2$}.\quad
    In this scenario we get $t \equiv_\beta 1$, hence this case is solved in absolutely the same way as the previous case $t-1 \equiv_\beta 0$.

    \bigskip\noindent
    \emph{Case $t-1 \equiv_\beta 2t-1$}.\quad
    Here, we immediately get $t \equiv_\beta 0$. By virtue of Table \ref{some_remainders}, we see that the element $0$ has a unique remainder modulo $\beta$ in the set $L'_t$. On the other hand, the set $L''_t$ contains no element with a unique remainder modulo $\beta$. In fact, we can group the elements of $L''_t$ into six equivalence pairs according to their remainders modulo $\beta$. We will rely on this fact later on.

    \bigskip\noindent
    \emph{Case $t-1 \equiv_\beta 2t+1$}.\quad
    In this case, we obtain $t \equiv_\beta -2$. It is easy to see from Table~\ref{some_remainders} that whenever $\beta \ge 10$, the element $t-2$ must have a unique remainder modulo $\beta$ within the set $L'_t$. Similarly, if $\beta \ge 7$, then the element $t+1$ surely has a unique remainder modulo $\beta$ inside the set $L''_t$.
\end{proof}

Now, by implementing Lemma \ref{unique_remainder_1}, we are able to prove the following lemma regarding the divisibility of $Q_t(x)$ and $R_t(x)$ polynomials by certain polynomials that shall be of later use to us.

\newpage
\begin{lemma}\label{weird_lemma_1}
    For any even $t \ge 6$ and each $\beta \ge 10$, the polynomial $Q_t(x)$ cannot be divisible by a polynomial $V(x) \in \mathbb{Q}[x]$ with at least two non-zero terms such that all of its terms have powers divisible by $\beta$. Similarly, for any even $t \ge 6$ and each $\beta \ge 7$, the polynomial $R_t(x)$ also cannot be divisible by any such $V(x)$.
\end{lemma}
\begin{proof}
    Let $t \ge 6$ be an arbitrarily chosen even integer and let $\beta \ge 10$ be any positive integer. Suppose that the polynomial $Q_t(x)$ is divisible by a $V(x)$ with at least two non-zero terms such that all of its terms have powers divisible by $\beta$. Now, we will use $Q_t^{(\beta, j)}(x)$ to denote the polynomial composed of all the terms of $Q_t(x)$ whose powers are congruent to $j$ modulo $\beta$, for each $j = \overline{0, \beta-1}$. If we write
    \[
        Q_t(x) = V(x) \, V_1(x)
    \]
    and use the notation $V_1^{(\beta, j)}(x)$ in a manner analogous to the previously stated $Q_t^{(\beta, j)}(x)$, it becomes easy to notice that
    \[
        Q_t^{(\beta, j)}(x) = V(x) \, V_1^{(\beta, j)}(x)
    \]
    must hold for each $j = \overline{0, \beta-1}$. Hence, we obtain that $V(x) \mid Q_t^{(\beta, j)}(x)$ is true for each $j = \overline{0, \beta-1}$. However, by virtue of Lemma \ref{unique_remainder_1}, we know that there exists an element of $L'_t$ that has a unique remainder modulo $\beta$ within the set. This implies that there exists a $j$ such that $Q_t^{(\beta, j)}(x)$ has the form $c \, x^a$ for some $c \in \mathbb{Z} \setminus \{ 0 \}$ and $a \in \mathbb{N}_0$. By taking this into consideration, we get that $V(x) \mid c \, x^a$, which further implies that $V(x)$ cannot have more than one non-zero term, thus yielding a contradiction.

    Now, let $t \ge 6$ be any even integer and let $\beta \ge 7$ be some positive integer. Suppose that $R_t(x)$ is divisible by a $V(x)$ with at least two non-zero terms such that all of its terms have powers divisible by $\beta$. If $\beta \nmid t$, then we can obtain a contradiction by applying Lemma \ref{unique_remainder_1} in a manner that is entirely analogous to the technique previously used while dealing with the $Q_t(x)$ polynomial. Thus, we choose to omit the proof details of this case and focus solely on the remaining scenario when $\beta \mid t$ holds.

    By taking into consideration the remainders given in Table \ref{some_remainders}, we see that for $\beta \ge 7$ and $\beta \mid t$, the divisibility $V(x) \mid R_t(x)$ further implies
    \begin{align*}
        V(x) &\mid 4x^t - 2,\\
        V(x) &\mid x^{t+2} + 2x^2,
    \end{align*}
    from which we swiftly obtain
    \begin{alignat*}{2}
        && V(x) &\mid x^2 \, (4x^t - 2) + (x^{t+2} + 2x^2)\\
        \implies \quad && V(x) &\mid 5x^{t+2} ,
    \end{alignat*}
    thus yielding a contradiction once more, as desired.
\end{proof}

We now turn our attention to the cyclotomic polynomials and investigate the divisibility of $Q_t(x)$ and $R_t(x)$ by these polynomials, for all possible even values $t \ge 6$. By taking into consideration that each cyclotomic polynomial $\Phi_b(x)$ must have at least two non-zero terms, it becomes apparent that Lemma \ref{weird_lemma_1} will play a big role in our analysis to come. In fact, its usage is immediately demonstrated within the next lemma.

\begin{lemma}\label{weird_lemma_2}
    For each even $t \ge 6$, the divisibility $\Phi_b(x) \mid Q_t(x)$ for some $b \in \mathbb{N}$ implies
    \begin{itemize}
        \item $p^2 \nmid b$ for any prime number $p \ge 11$;
        \item $7^3 \nmid b$, $5^3 \nmid b$, $3^4 \nmid b$, $2^2 \nmid b$.
    \end{itemize}
    Also, for each even $t \ge 6$, the divisibility $\Phi_b(x) \mid R_t(x)$ for some $b \in \mathbb{N}$ implies
    \begin{itemize}
        \item $p^2 \nmid b$ for any prime number $p \ge 7$;
        \item $5^3 \nmid b$, $3^3 \nmid b$, $2^2 \nmid b$.
    \end{itemize}
\end{lemma}
\begin{proof}
    Let $b \in \mathbb{N}$ be such that $p^2 \mid b$ for some prime number $p$. In this case, $\frac{b}{p}$ is a positive integer divisible by $p$, hence we get that $\Phi_b(x) = \Phi_{\frac{b}{p}}(x^p)$ (see, for example, \cite[p.\ 160]{Nagell}). Similarly, if $p^k \mid b$ for some $k \ge 2$, we inductively obtain that $\Phi_b(x) = \Phi_{\frac{b}{p^{k-1}}}(x^{p^{k-1}})$. We will now implement this observation in order to complete the proof of the lemma by dividing it into two separate cases for $Q_t(x)$ and $R_t(x)$.

    \bigskip\noindent
    \emph{Case $Q_t(x)$}.\quad
    Suppose that $\Phi_b(x) \mid Q_t(x)$ for some even $t \ge 6$ and some $b \in \mathbb{N}$. If $p^2 \mid b$ for some prime number $p \ge 11$, we then get that $\Phi_b(x) = \Phi_{\frac{b}{p}}(x^p)$, hence all the terms of $\Phi_b(x)$ must have powers divisible by $p \ge 11$. By virtue of Lemma \ref{weird_lemma_1}, the divisibility $\Phi_b(x) \mid Q_t(x)$ cannot hold, hence we obtain a contradiction.

    If we suppose that $7^3 \mid b$ or $5^3 \mid b$ or $3^4 \mid b$, we get that all the terms of $\Phi_b(x)$ must have powers divisible by $49$ or $25$ or $27$, respectively. In each of these cases, Lemma \ref{weird_lemma_1} tells us that the divisibility $\Phi_b(x) \mid Q_t(x)$ does not hold, thus yielding a contradiction. In order to prove the part of the lemma regarding the $Q_t(x)$ polynomial, it becomes sufficient to show that $4 \mid b$ cannot be true.

    Now, suppose that $4 \mid b$ holds. In this case, we immediately get that $\Phi_b(x)$ contains only terms whose powers are even. By taking into consideration that the numbers $0, 2, t-2, t+2, 2t-2$ are even, while $3, t-1, t+3, 2t-1, 2t+1$ are odd, we conclude that
    \begin{align*}
        \Phi_b(x) &\mid 2x^{2t-2} - x^{t+2} - x^{t-2} + 2x^2 - 2 ,\\
        \Phi_b(x) &\mid 2x^{2t+1} - 2x^{2t-1} + x^{t+3} + x^{t-1} - 2x^3 .
    \end{align*}
    If we denote
    \begin{align*}
        A(x) &= 2x^{2t-2} - x^{t+2} - x^{t-2} + 2x^2 - 2,\\
        B(x) &= 2x^{2t+1} - 2x^{2t-1} + x^{t+3} + x^{t-1} - 2x^3,\\
        C(x) &= x^{t+7} - x^{t+5} + x^{t+3} - x^{t+1} + \frac{1}{2}x^9 + 2x^7 - 3x^5 + 4x^3 - \frac{3}{2}x,\\
        D(x) &= -x^{t+4}-x^t+\frac{1}{2}x^8-x^4+2x^2-\frac{3}{2},
    \end{align*}
    then it can be further obtained that
    \begin{alignat*}{2}
        && \Phi_b(x) &\mid A(x) \, C(x) + B(x) \, D(x)\\
        \implies \quad && \Phi_b(x) &\mid 3x^9 - 8x^7 + 10x^5 - 8x^3 + 3x\\
        \implies \quad && \Phi_b(x) &\mid x (x-1)^2 (x+1)^2 (3x^4 - 2x^2 + 3)\\
        \implies \quad && \Phi_b(x) &\mid 3x^4 - 2x^2 + 3 .
    \end{alignat*}
    However, the polynomial $3x^4 - 2x^2 + 3 \in \mathbb{Q}[x]$ has no roots of unity among its roots, as demonstrated in Appendix \ref{problematic_roots}, thus yielding a contradiction. This means that $4 \mid b$ cannot possibly be true, as desired.

    \bigskip\noindent
    \emph{Case $R_t(x)$}.\quad
    Suppose that $\Phi_b(x) \mid R_t(x)$ for some even $t \ge 6$ and some $b \in \mathbb{N}$. It can be shown that $p^2 \nmid b$ for any prime $p \ge 7$, as well as $5^3 \nmid b$ and $3^3 \nmid b$, by implementing Lemma \ref{weird_lemma_1} in a completely analogous manner as done in the proof of the previous case. For this reason, we choose to leave out the according details. Thus, in order to finalize the proof, it is enough to show that $4 \nmid b$.

    Suppose that $4 \mid b$ does hold. Similarly as in the previous case, we conclude that $\Phi_b(x)$ contains only terms whose powers are even. Besides that, the numbers $0, 2, t-2, t, t+2, 2t-2$ are even, while $3, t-1, t+1, t+3, 2t-1, 2t+1$ are odd. Bearing this in mind, we get
    \begin{align*}
        \Phi(b) &\mid 2x^{2t-2}  + x^{t+2} + 4x^t + x^{t-2} + 2x^2 - 2 ,\\
        \Phi(b) &\mid 2x^{2t+1} - 2x^{2t-1} - x^{t+3} - 4x^{t+1} - x^{t-1} - 2x^3 .
    \end{align*}
    Now, if we denote
    \begin{align*}
        A(x) &= 2x^{2t-2}  + x^{t+2} + 4x^t + x^{t-2} + 2x^2 - 2 ,\\
        B(x) &= 2x^{2t+1} - 2x^{2t-1} - x^{t+3} - 4x^{t+1} - x^{t-1} - 2x^3 ,\\
        C(x) &= -x^{t+7} - 3x^{t+5} + 3x^{t+3} + x^{t+1} + \frac{1}{2} x^9 + 6x^7 + 5x^5 + 8x^3 - \frac{3}{2}x,\\
        D(x) &= x^{t+4} + 4x^{t+2} + x^t + \frac{1}{2} x^8 + 4x^6 + 7x^4 + 6x^2 - \frac{3}{2},
    \end{align*}
    it is clear that
    \begin{alignat*}{2}
        && \Phi_b(x) &\mid A(x) \, C(x) + B(x) \, D(x)\\
        \implies \quad && \Phi_b(x) &\mid 3x^9 - 16x^7 - 6x^5 - 16x^3 + 3x\\
        \implies \quad && \Phi_b(x) &\mid x (x^2 - 2x - 1) (x^2 + 2x - 1) (3x^4 + 2x^2 + 3)\\
        \implies \quad && \Phi_b(x) &\mid (x^2 - 2x - 1) (x^2 + 2x - 1) (3x^4 + 2x^2 + 3) .
    \end{alignat*}
    However, neither of the polynomials $x^2 - 2x - 1, x^2 + 2x - 1, 3x^4 + 2x^2 + 3$ contains a root that represents a root of unity, which immediately leads us to a contradiction. Thus, we reach $4 \nmid b$.
\end{proof}

Lemma \ref{weird_lemma_2} indicates that the cyclotomic polynomials which divide $Q_t(x)$ and $R_t(x)$ are very specific. Moreover, it can be shown that for any $b \ge 3$, the cyclotomic polynomial $\Phi_b(x)$ can divide neither $Q_t(x)$ nor $R_t(x)$. In fact, our next step shall be to prove this exact statement. In order to do this, we will need the following two short auxiliary lemmas.

\begin{lemma}\label{p_2p_1}
    For each even $t \ge 6$ and each prime number $p \ge 11$, $Q_t(x)$ cannot be divisible by $\Phi_p(x)$ or $\Phi_{2p}(x)$. Similarly, for any even $t \ge 6$ and any prime number $p \ge 13$, $R_t(x)$ cannot be divisible by $\Phi_p(x)$ or $\Phi_{2p}(x)$.
\end{lemma}
\begin{proof}
    The lemma statement about the $Q_t(x)$ polynomial can be proved in a fairly analogous manner as the part regarding $R_t(x)$. In fact, the proof of $\Phi_p(x) \nmid R_t(x)$ and $\Phi_{2p}(x) \nmid R_t(x)$ is slightly more difficult to perform due to the existence of one additional edge case which does not exist when we are dealing with $Q_t(x)$. For this reason, we choose to leave out the proof details for $\Phi_p(x) \nmid Q_t(x)$ and $\Phi_{2p}(x) \nmid Q_t(x)$ and focus solely on the $R_t(x)$ polynomial.

    Now, let $t \ge 6$ be an arbitrary even integer and let $p \ge 11$ be some prime number. By noticing that
    \begin{align*}
        \Phi_p(x) = \sum_{j = 0}^{p-1} x^j, && \Phi_{2p}(x) = \sum_{j = 0}^{p-1} (-1)^j x^j ,
    \end{align*}
    we immediately see that $\deg \Phi_p = \deg \Phi_{2p} = p-1$. We will finalize the proof by splitting the problem into two cases depending on whether we are dealing with $\Phi_p(x)$ or $\Phi_{2p}(x)$.

    \bigskip\noindent
    \emph{Case $\Phi_p(x)$}.\quad
    Let $R_t^{\bmod p}(x)$ be the following polynomial:
    \begin{align*}
        R_t^{\bmod p}(x) = 2x^{(2t+1) \bmod p} &- 2x^{(2t-1) \bmod p} + 2x^{(2t-2) \bmod p} - x^{(t+3) \bmod p}\\
        &+ x^{(t+2) \bmod p} - 4x^{(t+1) \bmod p} + 4x^{t \bmod p}\\
        &- x^{(t-1) \bmod p} + x^{(t-2) \bmod p} - 2x^3 + 2x^2 - 2.
    \end{align*}
    It is clear that $\Phi_p(x) \mid R_t(x)$ holds if and only if $\Phi_p(x) \mid R_t^{\bmod p}(x)$ does, too. If we suppose that $\Phi_p(x) \mid R_t(x)$ is true and take into consideration that
    \[
        \deg R_t^{\bmod p} \le p-1 = \deg \Phi_p ,
    \]
    we quickly obtain two possibilities:
    \begin{itemize}
        \item $R_t^{\bmod p}(x) \equiv 0$;
        \item $R_t^{\bmod p}(x) = c \, \Phi_p(x)$ for some $c \in \mathbb{Q} \setminus \{ 0 \}$.
    \end{itemize}

    It is not difficult to see that $R_t^{\bmod p}(x) \equiv 0$ cannot hold. If $p \nmid t$, then Lemma~\ref{unique_remainder_1} dictates that there exists an element of $L''_t$ that has a unique remainder modulo $p$ within that set. Hence, $R_t^{\bmod p}(x)$ must have at least one term corresponding to that element. On the other hand, if $p \mid t$, then according to Table \ref{some_remainders}, in order for $R_t^{\bmod p}(x) \equiv 0$ to be true, we would need $4x^{t \bmod p} - 2 = 0$ to hold, which clearly does not. It is worth pointing out that while performing the analogous proof for $Q_t(x)$, the edge case $p \mid t$ does not exist, which can immediately be noticed in the formulation itself of Lemma \ref{unique_remainder_1}.

    Now, suppose that $R_t^{\bmod p}(x) = c \, \Phi_p(x)$ holds for some $c \in \mathbb{Q} \setminus \{ 0 \}$. The polynomial $\Phi_p(x)$ has $p-1$ non-zero terms, which means that $R_t^{\bmod p}(x)$ needs to have exactly $p-1$ non-zero terms as well. This is obviously not possible whenever $p > 13$, since $R_t^{\bmod p}(x)$ can have at most $12$ non-zero terms. As for the scenario when $p = 13$, it is easy to see that $R_t^{\bmod p}(x)$ has $12$ non-zero terms if and only if all the elements of the set $L''_t$ have mutually distinct remainders modulo $p$. However, in that case we see that $R_t^{\bmod p}(x) = c \, \Phi_p(x)$ cannot possibly be true due to the fact that the coefficients corresponding to the non-zero terms of $\Phi_p(x)$ are all equal, while that would clearly not be the case for $R_t^{\bmod p}(x)$.

    \bigskip\noindent
    \emph{Case $\Phi_{2p}(x)$}.\quad
    Let $\hat{R}_t^{\bmod p}(x)$ be the following polynomial:
    \begin{align*}
        \hat{R}_t^{\bmod p}(x) = 2 & (-1)^{\lfloor\frac{2t+1}{p}\rfloor} x^{(2t+1) \bmod p} - 2 (-1)^{\lfloor\frac{2t-1}{p}\rfloor} x^{(2t-1) \bmod p}\\
        &+ 2 (-1)^{\lfloor\frac{2t-2}{p}\rfloor} x^{(2t-2) \bmod p} - (-1)^{\lfloor\frac{t+3}{p}\rfloor} x^{(t+3) \bmod p}\\
        &+ (-1)^{\lfloor\frac{t+2}{p}\rfloor} x^{(t+2) \bmod p} - 4 (-1)^{\lfloor\frac{t+1}{p}\rfloor} x^{(t+1) \bmod p}\\
        &+ 4 (-1)^{\lfloor\frac{t}{p}\rfloor} x^{t \bmod p} - (-1)^{\lfloor\frac{t-1}{p}\rfloor} x^{(t-1) \bmod p}\\
        &+ (-1)^{\lfloor\frac{t-2}{p}\rfloor} x^{(t-2) \bmod p} - 2x^3 + 2x^2 - 2.
    \end{align*}
    We know that each primitive $2p$-th root of unity gives $-1$ when raised to the power of $p$. For this reason, it is not difficult to conclude that $\Phi_{2p}(x) \mid R_t(x)$ is equivalent to $\Phi_{2p}(x) \mid \hat{R}_t^{\bmod p}(x)$. Now, if we suppose that $\Phi_{2p}(x) \mid R_t(x)$ holds and bear in mind that
    \[
        \deg \hat{R}_t^{\bmod p} \le p-1 = \deg \Phi_{2p} ,
    \]
    we reach the same two possibilities as in the previous case:
    \begin{itemize}
        \item $\hat{R}_t^{\bmod p}(x) \equiv 0$;
        \item $\hat{R}_t^{\bmod p}(x) = c \, \Phi_p(x)$ for some $c \in \mathbb{Q} \setminus \{ 0 \}$.
    \end{itemize}
    The rest of the proof can be carried out in a manner analogous to the previous case. Thus, we choose to omit it.
\end{proof}

\begin{lemma}\label{concrete_b_1}
    For each even $t \ge 6$, $Q_t(x)$ cannot be divisible by a cyclotomic polynomial $\Phi_b(x)$ where $b \ge 3$ is a positive integer such that
    \begin{itemize}
        \item it does not have any prime factors outside of the set $\{2, 3, 5, 7\}$;
        \item it does not contain all the prime factors from the set $\{3, 5, 7 \}$;
        \item $2^2 \nmid b$, $3^4 \nmid b$, $5^3 \nmid b$, $7^3 \nmid b$.
    \end{itemize}
    Also, for any even $t \ge 6$, $R_t(x)$ cannot be divisible by a cyclotomic polynomial $\Phi_b(x)$ where $b \ge 3$ is a positive integer such that
    \begin{itemize}
        \item it does not have any prime factors outside of the set $\{2, 3, 5, 7, 11\}$;
        \item it does not contain both prime factors from the set $\{7, 11\}$ or from the set $\{5, 11\}$;
        \item $2^2 \nmid b$, $3^3 \nmid b$, $5^3 \nmid b$, $7^2 \nmid b$, $11^2 \nmid b$.
    \end{itemize}
\end{lemma}
\begin{proof}
    Let $Q_t^{\bmod b}(x)$ and $R_t^{\bmod b}(x)$ be the next two polynomials:
    \begin{align*}
        Q_t^{\bmod b}(x) = 2&x^{(2t+1) \bmod b} - 2x^{(2t-1) \bmod b} + 2x^{(2t-2) \bmod b} + x^{(t+3) \bmod b}\\
        &- x^{(t+2) \bmod b} + x^{(t-1) \bmod b} - x^{(t-2) \bmod b}\\
        &- 2x^{3 \bmod b} + 2x^{2 \bmod b} - 2,\\
        R_t^{\bmod b}(x) = 2&x^{(2t+1) \bmod b} - 2x^{(2t-1) \bmod b} + 2x^{(2t-2) \bmod b} - x^{(t+3) \bmod b}\\
        &+ x^{(t+2) \bmod b} - 4x^{(t+1) \bmod b} + 4x^{t \bmod b} - x^{(t-1) \bmod b}\\
        &+ x^{(t-2) \bmod b} - 2x^{3 \bmod b} + 2x^{2 \bmod b} - 2.
    \end{align*}
    Here, it is crucial to point out that $\Phi_b(x) \mid Q_t(x) \iff \Phi_b(x) \mid Q_t^{\bmod b}(x)$ and $\Phi_b(x) \mid R_t(x) \iff \Phi_b(x) \mid R_t^{\bmod b}(x)$. However, for a fixed value of $b \in \mathbb{N}$, it is clear that there exist only finitely many polynomials $Q_t^{\bmod b}(x)$ and $R_t^{\bmod b}(x)$ as $t$ ranges over the even integers greater than or equal to $6$. Thus, if we show that $\Phi_b(x)$ divides none of these concrete $Q_t^{\bmod b}(x)$ polynomials, this is sufficient to prove that $\Phi_b(x)$ does not divide $Q_t(x)$. The same can be said regarding $R_t(x)$.

    In order to prove the lemma, it is enough to demonstrate that $\Phi_b(x) \nmid Q_t^{\bmod b}(x)$ and $\Phi_b(x) \nmid R_t^{\bmod b}(x)$ for all the required values of $b$ and for all the possible remainders $t \bmod b$. However, the lemma formulation specifies only a finite set of $b$ values corresponding to $Q_t(x)$ and to $R_t(x)$. For this reason, it is trivial to perform the proof of the lemma via computer. The required computational results are disclosed in Appendices \ref{qt_inspection} and \ref{rt_inspection}.
\end{proof}

We now proceed to prove the aforementioned statement regarding the divisibility of $Q_t(x)$ and $R_t(x)$ polynomials by cyclotomic polynomials. In order to do this, we shall heavily rely on Theorem \ref{filaseta}. The reason why this theorem is so convenient to use is clear --- it is due to the sheer fact that the $Q_t(x)$ and $R_t(x)$ polynomials have very few non-zero terms.

\begin{lemma}\label{qt_cyclotomic}
    For each even $t \ge 6$, $Q_t(x)$ is not divisible by any cyclotomic polynomial $\Phi_b(x)$ where $b \ge 3$.
\end{lemma}
\begin{proof}
    Suppose that $\Phi_b(x) \mid Q_t(x)$ for some even $t \ge 6$ and some $b \ge 3$. We now divide the problem into two cases depending on whether $b$ is divisible by some prime number from the set $\{3, 5, 7\}$.

    \bigskip\noindent
    \emph{Case $3 \nmid b \land 5 \nmid b \land 7 \nmid b$}.\quad
    In this case, it is clear that $b$ has at least one prime factor greater than $7$, since $b \notin \{1, 2\}$ and $2^2 \nmid b$, according to Lemma \ref{weird_lemma_2}. Now, since $Q_t(x)$ has exactly $10$ non-zero terms, and $p - 2 \ge 10 - 2$ for any prime $p \ge 11$, we are able to repeatedly apply Theorem \ref{filaseta} in order to cancel out any additional prime divisor of $b$ greater than $7$, until exactly one is left. This leads us to
    \[
        \Phi_{b'}(x) \mid Q_t(x) ,
    \]
    where $b'$ has a single prime divisor greater than $7$, and is potentially divisible by two as well. Taking into consideration Lemma \ref{weird_lemma_2}, we conclude that $b'$ must either be equal to some prime $p \ge 11$ or have the form $2p$, where $p \ge 11$ is a prime number. Either way, Lemma \ref{p_2p_1} tells us that such a $\Phi_{b'}(x)$ cannot possibly divide $Q_t(x)$, hence we obtain a contradiction.

    \bigskip\noindent
    \emph{Case $3 \mid b \lor 5 \mid b \lor 7 \mid b$}.\quad
    In this scenario, we can apply Theorem \ref{filaseta} in a similar fashion in order to cancel out any potential prime divisor of $b$ greater than $7$ until we reach
    \[
        \Phi_{b'}(x) \mid Q_t(x) ,
    \]
    where $b' \ge 3$ is such that all of its prime factors belong to the set $\{2, 3, 5, 7\}$, and $2^2 \nmid b'$, $3^4 \nmid b'$, $5^3 \nmid b'$, $7^3 \nmid b'$, by virtue of Lemma \ref{weird_lemma_2}. Furthermore, we know that $(3-2)+(5-2)+(7-2)>10-2$, hence we can suppose without loss of generality that $b'$ is not divisible by at least one prime number from the set $\{3, 5, 7\}$, in accordance with Theorem \ref{filaseta}. However, Lemma \ref{concrete_b_1} dictates that such a $\Phi_{b'}(x)$ cannot divide $Q_t(x)$, yielding a contradiction.
\end{proof}

\begin{lemma}\label{rt_cyclotomic}
    For each even $t \ge 6$, $R_t(x)$ is not divisible by any cyclotomic polynomial $\Phi_b(x)$ where $b \ge 3$.
\end{lemma}
\begin{proof}
    Suppose that $\Phi_b(x) \mid R_t(x)$ for some even $t \ge 6$ and some $b \ge 3$. We proceed by dividing the problem into two cases depending on whether $b$ is divisible by some prime number from $\{3, 5, 7, 11\}$.

    \bigskip\noindent
    \emph{Case $3 \nmid b \land 5 \nmid b \land 7 \nmid b \land 11 \nmid b$}.\quad
    Due to the fact that $b \not\in \{1, 2\}$ and $2^2 \nmid b$, by virtue of Lemma \ref{weird_lemma_2}, we deduce that $b$ has at least one prime factor greater than $11$. Since $R_t(x)$ has exactly $12$ non-zero terms and $p - 2 > 12 - 2$ for any prime $p \ge 13$, we can implement Theorem \ref{filaseta} and Lemma \ref{weird_lemma_2} in an analogous fashion as in Lemma \ref{qt_cyclotomic} in order to reach
    \[
        \Phi_{b'}(x) \mid R_t(x) ,
    \]
    for some $b'$ that is either equal to a prime $p \ge 13$ or has the form $2p$, for some prime number $p \ge 13$. Once again, Lemma \ref{p_2p_1} tells us that such a divisibility cannot hold, leading to a contradiction.

    \bigskip\noindent
    \emph{Case $3 \mid b \lor 5 \mid b \lor 7 \mid b \lor 11 \mid b$}.\quad
    In this case, we apply Theorem \ref{filaseta} once more in order to cancel out any potential prime divisor of $b$ greater than $11$ until we get
    \[
        \Phi_{b'}(x) \mid R_t(x),
    \]
    where $b' \ge 3$ is such that all of its prime factors belong to the set $\{2, 3, 5, 7, 11\}$, and $2^2 \nmid b'$, $3^3 \nmid b'$, $5^3 \nmid b'$, $7^2 \nmid b'$, $11^2 \nmid b'$ due to Lemma \ref{weird_lemma_2}. Besides that, it is clear that $(7-2) + (11-2) > 12 - 2$ and $(5-2) + (11-2) > 12 - 2$, which means that it is safe to suppose that $b'$ is not divisible by both elements from $\{7, 11\}$ or from $\{5, 11\}$. Bearing this in mind, it is easy to reach a contradiction by taking into consideration Lemma \ref{concrete_b_1}.
\end{proof}

Finally, we are able to put all the pieces of the puzzle together and complete the proof of Theorem \ref{main_th_1}.

\bigskip\noindent
\emph{Proof of Theorem \ref{main_th_1}}.\quad
From Eq.\ (\ref{polynomial_formula}) we immediately obtain
\[
    P(\zeta) = \left( \zeta^{\frac{n}{4}} + \frac{1}{\zeta^{\frac{n}{4}}} \right) + \left( \zeta^{\frac{n}{4}+2} + \frac{1}{\zeta^{\frac{n}{4}+2}} \right) + \sum_{\substack{j=1,\\ j \neq t-2}}^{t}\left( \zeta^j + \frac{1}{\zeta^j}\right) + \sum_{\substack{j=\frac{n}{2}-t,\\ j \neq \frac{n}{2} - t+2}}^{\frac{n}{2}-1}\left( \zeta^j + \frac{1}{\zeta^j}\right),
\]
where $\zeta$ is an arbitrarily chosen $n$-th root of unity different from $1$ and $-1$. It is easy to further conclude that
\begin{equation}\label{aux_5}
    P(\zeta) = \left( \zeta^{\frac{n}{4}} + \frac{1}{\zeta^{\frac{n}{4}}} \right) + \left( \zeta^{\frac{n}{4}+2} + \frac{1}{\zeta^{\frac{n}{4}+2}} \right) + \sum_{\substack{j=1,\\ j \neq t-2}}^{t}\left( \zeta^j + \frac{1}{\zeta^j} + \zeta^{\frac{n}{2} - j} + \frac{1}{\zeta^{\frac{n}{2}-j}} \right) .
\end{equation}
We will finish the proof by showing that $P(\zeta) \neq 0$ must necessarily hold. For the purpose of making the proof easier to follow, we shall divide it into two cases depending on whether $\zeta^\frac{n}{2}$ is equal to $1$ or $-1$.

\bigskip\noindent
\emph{Case $\zeta^\frac{n}{2} = -1$}.\quad
It is straightforward to see that $\zeta^{\frac{n}{2} - j} = -\dfrac{1}{\zeta^j}$ and $\dfrac{1}{\zeta^{\frac{n}{2} - j}} = -\zeta^j$, which swiftly gives
\[
    \zeta^j + \frac{1}{\zeta^j} + \zeta^{\frac{n}{2} - j} + \frac{1}{\zeta^{\frac{n}{2}-j}} = 0
\]
for any $j = \overline{1, t}$. Thus, Eq.\ (\ref{aux_5}) simplifies to
\[
    P(\zeta) = \zeta^{\frac{n}{4}+2} + \zeta^{\frac{n}{4}} + \frac{1}{\zeta^{\frac{n}{4}}} + \frac{1}{\zeta^{\frac{n}{4}+2}} .
\]
The condition $P(\zeta) = 0$ now becomes equivalent to
\allowdisplaybreaks
\begin{alignat*}{2}
    && P(\zeta) &= 0\\
    \iff \quad && \zeta^{\frac{n}{4}+2} \left( \zeta^{\frac{n}{4}+2} + \zeta^{\frac{n}{4}} + \frac{1}{\zeta^{\frac{n}{4}}} + \frac{1}{\zeta^{\frac{n}{4}+2}} \right) &=0\\
    \iff \quad && \zeta^{\frac{n}{2} + 4} + \zeta^{\frac{n}{2} + 2} + \zeta^2 + 1 &= 0\\
    \iff \quad && -\zeta^4 - \zeta^2 + \zeta^2 + 1 &= 0\\
    \iff \quad && \zeta^4 &= 1 .
\end{alignat*}
However, due to the fact that $4 \mid \frac{n}{2}$, it is evident that each fourth root of unity among $\zeta$ cannot satisfy $\zeta^\frac{n}{2} = -1$. Thus, provided $\zeta^\frac{n}{2} = -1$, $\zeta^4 \neq 1$ cannot be true, from which we immediately obtain $P(\zeta) \neq 0$, as desired.

\bigskip\noindent
\emph{Case $\zeta^\frac{n}{2} = 1$}.\quad
Here, we swiftly obtain $\zeta^{\frac{n}{2} - j} = \dfrac{1}{\zeta^j}$ and $\dfrac{1}{\zeta^{\frac{n}{2} - j}} = \zeta^j$. This immediately implies
\[
    \zeta^j + \frac{1}{\zeta^j} + \zeta^{\frac{n}{2} - j} + \frac{1}{\zeta^{\frac{n}{2}-j}} = 2 \left( \zeta^j + \frac{1}{\zeta^j} \right)
\]
for any $j = \overline{1, t}$. By applying Eq.\ (\ref{aux_5}), it is now clear that $P(\zeta) = 0$ is equivalent to
\begin{alignat*}{2}
    && P(\zeta) &= 0\\
    \iff \quad && \left( \zeta^{\frac{n}{4}+2} + \zeta^{\frac{n}{4}} + \frac{1}{\zeta^{\frac{n}{4}}} + \frac{1}{\zeta^{\frac{n}{4}+2}} \right) + 2 \sum_{\substack{j=1,\\ j \neq t-2}}^{t}\left( \zeta^j + \frac{1}{\zeta^j} \right) &= 0\\
    \iff \quad && \left( \zeta^{\frac{n}{4}+2} + \zeta^{\frac{n}{4}} + \frac{1}{\zeta^{\frac{n}{4}}} + \frac{1}{\zeta^{\frac{n}{4}+2}} \right) - 2 - 2 \zeta^{t-2} - \frac{2}{\zeta^{t-2}} + 2 \sum_{j=-t}^{t} \zeta^j &= 0\\
    \iff \quad && \zeta^{t} \left( \zeta^{\frac{n}{4}+2} + \zeta^{\frac{n}{4}} + \frac{1}{\zeta^{\frac{n}{4}}} + \frac{1}{\zeta^{\frac{n}{4}+2}} - 2 - 2 \zeta^{t-2} - \frac{2}{\zeta^{t-2}} + 2 \sum_{j=-t}^{t} \zeta^j \right) &= 0\\
    \iff \quad && \zeta^{t+\frac{n}{4}+2} + \zeta^{t+\frac{n}{4}} + \zeta^{t-\frac{n}{4}} + \zeta^{t-\frac{n}{4}-2} - 2\zeta^{t} - 2\zeta^{2t-2} - 2\zeta^2 + 2 \sum_{j=0}^{2t} \zeta^j &= 0 .
\end{alignat*}
Given the fact that
\begin{align*}
    (\zeta &- 1) \left( \zeta^{t+\frac{n}{4}+2} + \zeta^{t+\frac{n}{4}} + \zeta^{t-\frac{n}{4}} + \zeta^{t-\frac{n}{4}-2} - 2\zeta^{t} - 2\zeta^{2t-2} - 2\zeta^2 + 2 \sum_{j=0}^{2t} \zeta^j \right) =\\
    &= \zeta^{t + \frac{n}{4} + 3} - \zeta^{t + \frac{n}{4} + 2} + \zeta^{t + \frac{n}{4} + 1} - \zeta^{t + \frac{n}{4}} + \zeta^{t - \frac{n}{4} + 1} - \zeta^{t - \frac{n}{4}} + \zeta^{t - \frac{n}{4} - 1} - \zeta^{t - \frac{n}{4} - 2}\\
    &\qquad - 2\zeta^{t+1} + 2\zeta^t - 2\zeta^{2t-1} + 2\zeta^{2t-2} - 2\zeta^3 + 2\zeta^2 + 2\zeta^{2t+1} - 2\\
    &= \zeta^{t + \frac{n}{4} + 3} - \zeta^{t + \frac{n}{4} + 2} + 2\zeta^{t + \frac{n}{4} + 1} - 2\zeta^{t + \frac{n}{4}} + \zeta^{t + \frac{n}{4} - 1} - \zeta^{t + \frac{n}{4} - 2}\\
    &\qquad + 2\zeta^{2t+1} - 2\zeta^{2t-1} + 2\zeta^{2t-2} - 2\zeta^{t+1} + 2\zeta^t - 2\zeta^3 + 2\zeta^2 - 2,
\end{align*}
it is straightforward to see that $P(\zeta) = 0$ is further equivalent to
\begin{align}\label{aux_6}
    \begin{split}
    &\zeta^{t + \frac{n}{4} + 3} - \zeta^{t + \frac{n}{4} + 2} + 2\zeta^{t + \frac{n}{4} + 1} - 2\zeta^{t + \frac{n}{4}} + \zeta^{t + \frac{n}{4} - 1} - \zeta^{t + \frac{n}{4} - 2}\\
    &\qquad + 2\zeta^{2t+1} - 2\zeta^{2t-1} + 2\zeta^{2t-2} - 2\zeta^{t+1} + 2\zeta^t - 2\zeta^3 + 2\zeta^2 - 2 = 0 .
    \end{split}
\end{align}
We now divide the problem into two separate subcases depending on whether $\zeta^\frac{n}{4}$ is equal to $1$ or $-1$.

\medskip\noindent
\emph{Subcase $\zeta^{\frac{n}{4}} = 1$}.\quad
In this subcase, it can be easily noticed from Eq.\ (\ref{aux_6}) that $P(\zeta) = 0$ is equivalent to
\begin{align*}
    &\zeta^{t + 3} - \zeta^{t + 2} + 2\zeta^{t + 1} - 2\zeta^{t} + \zeta^{t - 1} - \zeta^{t - 2}\\
    &\qquad + 2\zeta^{2t+1} - 2\zeta^{2t-1} + 2\zeta^{2t-2} - 2\zeta^{t+1} + 2\zeta^t - 2\zeta^3 + 2\zeta^2 - 2 = 0 ,
\end{align*}
that is
\begin{equation}\label{aux_7}
    2\zeta^{2t+1} - 2\zeta^{2t-1} + 2\zeta^{2t-2} + \zeta^{t + 3} - \zeta^{t + 2} + \zeta^{t - 1} - \zeta^{t - 2} - 2\zeta^3 + 2\zeta^2 - 2 = 0 .
\end{equation}
Suppose that $P(\zeta) = 0$ does hold for some $n$-th root of unity $\zeta$ different from $1$ and $-1$. For $t = 4$, Eq.\ (\ref{aux_7}) simplifies to
\begin{alignat*}{2}
    && 2\zeta^9 - \zeta^7 + \zeta^6 - \zeta^3 + \zeta^2 - 2 &= 0\\
    \iff \quad && (\zeta - 1)(\zeta + 1)^2 (2\zeta^6 - 2\zeta^5 + 3\zeta^4 - 2\zeta^3 + 3\zeta^2 - 2\zeta + 2) &= 0\\
    \iff \quad && 2\zeta^6 - 2\zeta^5 + 3\zeta^4 - 2\zeta^3 + 3\zeta^2 - 2\zeta + 2 &= 0 .
\end{alignat*}
\newpage\noindent
However, the polynomial $2x^6 - 2x^5 + 3x^4 - 2x^3 + 3x^2 - 2x + 2 \in \mathbb{Q}[x]$ has no roots of unity among its roots, as shown in Appendix \ref{problematic_roots}. Thus, we reach a contradiction. On the other hand, if $t \ge 6$, then Eq.\ (\ref{aux_7}) is equivalent to $Q_t(\zeta) = 0$, which immediately implies that the polynomial $Q_t(x)$ must be divisible by a cyclotomic polynomial $\Phi_b(x)$ where $b \ge 3$. However, by virtue of Lemma \ref{qt_cyclotomic}, this is not possible, yielding a contradiction once more.

\medskip\noindent
\emph{Subcase $\zeta^{\frac{n}{4}} = -1$}.\quad
Here, implementing Eq.\ (\ref{aux_6}) means that $P(\zeta) = 0$ is equivalent to
\begin{align*}
    &-\zeta^{t + 3} + \zeta^{t + 2} - 2\zeta^{t + 1} + 2\zeta^{t} - \zeta^{t - 1} + \zeta^{t - 2}\\
    &\qquad + 2\zeta^{2t+1} - 2\zeta^{2t-1} + 2\zeta^{2t-2} - 2\zeta^{t+1} + 2\zeta^t - 2\zeta^3 + 2\zeta^2 - 2 = 0 ,
\end{align*}
that is
\begin{align}\label{aux_8}
    \begin{split}
    &2\zeta^{2t+1} - 2\zeta^{2t-1} + 2\zeta^{2t-2} - \zeta^{t + 3} + \zeta^{t + 2}\\
    &\qquad - 4\zeta^{t+1} + 4\zeta^t - \zeta^{t - 1} + \zeta^{t - 2} - 2\zeta^3 + 2\zeta^2 - 2 = 0 .
    \end{split}
\end{align}
Now, suppose that $P(\zeta) = 0$ is true for some $n$-th root of unity $\zeta \neq 1, -1$. If $t = 4$, then Eq.\ (\ref{aux_8}) transforms to
\begin{alignat*}{2}
    && 2\zeta^9 - 3\zeta^7 + 3\zeta^6 - 4\zeta^5 + 4\zeta^4 - 3\zeta^3 + 3\zeta^2 - 2 &= 0\\
    \iff \quad && (\zeta - 1)(2\zeta^8 + 2\zeta^7 - \zeta^6 + 2\zeta^5 - 2\zeta^4 + 2\zeta^3 - \zeta^2 + 2\zeta + 2) &= 0\\
    \iff \quad && 2\zeta^8 + 2\zeta^7 - \zeta^6 + 2\zeta^5 - 2\zeta^4 + 2\zeta^3 - \zeta^2 + 2\zeta + 2 &= 0 .
\end{alignat*}
However, the polynomial $2x^8 + 2x^7 - x^6 + 2x^5 - 2x^4 + 2x^3 - x^2 + 2x + 2 \in \mathbb{Q}[x]$ has no roots of unity among its roots, as demonstrated in Appendix \ref{problematic_roots}, hence $P(\zeta) = 0$ leads to a contradiction, as desired. On the other hand, whenever $t \ge 6$, Eq.\ (\ref{aux_8}) becomes equivalent to $R_t(\zeta) = 0$, which further implies that $R_t(x)$ is divisible by some cyclotomic polynomial $\Phi_b(x)$ where $b \ge 3$. However, Lemma \ref{rt_cyclotomic} dictates that this is impossible, hence we reach a contradiction yet again. \hfill\qed

\section{Construction for \texorpdfstring{$n \equiv_8 4 \land n \ge 4t + 12$}{n mod 8 = 4 and n >= 4t + 12}}\label{section_5}

In this section we shall give a constructive proof of the existence of a $4t$-regular circulant nut graph of any order $n \in \mathbb{N}$ such that $n \ge 4t + 12$ and $n \equiv_8 4$, for any even $t \ge 4$. The proof will be given in the form of the next theorem.

\begin{theorem}\label{main_th_2}
    For any even $t \ge 4$ and any $n \ge 4t + 12$ such that $n \equiv_8 4$, the circulant graph $\mathrm{Circ}(n, S''_{t, n})$ where
    \[
        S''_{t, n} = \{1, 2, \ldots, t-1 \} \cup \left\{ \frac{n}{4} - 1, \frac{n}{4} + 3 \right\} \cup \left\{\frac{n}{2} - (t-1), \ldots, \frac{n}{2} - 2, \frac{n}{2} -1 \right\}
    \]
    must be a $4t$-regular circulant nut graph of order $n$.
\end{theorem}

We will show that Theorem \ref{main_th_2} holds by using a similar strategy as with Theorem \ref{main_th_1}. To begin with, it can be easily deduced that the set $S''_{t, n}$ is well defined, since $t -  1 < \frac{n}{4} - 1$ and $\frac{n}{4} + 3 < \frac{n}{2} - (t-1)$ for each even $t \ge 4$ and each $n \ge 4t + 12$ such that $n \equiv_8 4$. Besides that, the set $S''_{t, n}$ certainly contains equally many odd and even integers, all of which are positive and smaller than $\frac{n}{2}$. This means that, by implementing Lemma \ref{damnjanovic_lemma_1}, we can prove Theorem \ref{main_th_2} if we simply show that $P(x)$ has no $n$-th roots of unity among its roots, except potentially $1$ or $-1$.

To begin with, we shall define the following two polynomials
\begin{align*}
    U_t(x) &= 2x^{2t-1} + x^{t+3} - x^{t+2} + x^{t+1} - 3x^t + 3x^{t-1} - x^{t-2} + x^{t-3} - x^{t-4} - 2,\\
    W_t(x) &= 2x^{2t-1} - x^{t+3} + x^{t+2} - x^{t+1} - x^t + x^{t-1} + x^{t-2} - x^{t-3} + x^{t-4} - 2,
\end{align*}
for each even $t \ge 4$. Now, for $t \ge 6$ we have $2t-1 > t+3$ and $t-4 > 0$, while the equalities $2t - 1 = t+3$ and $t - 4 = 0$ hold for $t = 4$. This means that the polynomials $U_t(x)$ and $W_t(x)$ have exactly $10$ non-zero terms for any even $t \ge 6$, and $8$ non-zero terms in case $t = 4$. We will use $M_t$ to denote the set containing the powers of these terms, i.e.\
\[
    M_t = \{ 0, t-4, t-3, t-2, t-1, t, t+1, t+2, t+3, 2t-1 \} ,
\]
for each even $t \ge 4$. We are now able to present the following lemma that demonstrates a property of $M_t$ similar to the one displayed in Lemma \ref{unique_remainder_1} regarding the sets $L'_t$ and $L''_t$.
\begin{lemma}\label{unique_remainder_2}
    For each even $t \ge 4$ and each $\beta \in \mathbb{N},\, \beta \ge 6$, $M_t$ must contain an element whose remainder modulo $\beta$ is unique within the set.
\end{lemma}
\begin{proof}
    It is clear that the six consecutive integers $t-3, t-2, t-1, t, t+1, t+2$ must all have mutually distinct remainders modulo $\beta$ for any $\beta \ge 6$. Regardless of whether $t = 4$ or $t \ge 6$, it is easy to establish that at least two of these integers must have a distinct remainder modulo $\beta$ from all the elements of the set $\{ 0, t-4, t+3, 2t - 1\}$. Hence, these integers must have a unique remainder modulo $\beta$ within $M_t$ and the lemma statement follows swiftly from here.
\end{proof}

By relying on Lemma \ref{unique_remainder_2}, we can now prove another lemma regarding the divisibility of $U_t(x)$ and $W_t(x)$ polynomials that is analogous to Lemma \ref{weird_lemma_1}.
\begin{lemma}\label{weird_lemma_3}
    For any even $t \ge 4$ and each $\beta \ge 6$, neither $U_t(x)$ nor $W_t(x)$ can be divisible by a polynomial $V(x) \in \mathbb{Q}[x]$ with at least two non-zero terms such that all of its terms have powers divisible by $\beta$.
\end{lemma}
\begin{proof}
    This lemma can be proved in an absolutely analogous manner as Lemma~\ref{weird_lemma_1}. The only difference is that Lemma \ref{unique_remainder_2} is implemented in place of Lemma \ref{unique_remainder_1}. For this reason, we choose to omit the proof details.
\end{proof}

In a similar manner as done so in Section \ref{section_4}, we now investigate the divisibility of $U_t(x)$ and $W_t(x)$ polynomials by cyclotomic polynomials. By directly implementing Lemma \ref{weird_lemma_3}, we are able to prove the following result.
\begin{lemma}\label{weird_lemma_4}
    For each even $t \ge 4$, if $\Phi_b(x) \mid U_t(x)$ or $\Phi_b(x) \mid W_t(x)$ hold for some $b \ge 3$, we then necessarily have
    \begin{itemize}
        \item $p^2 \nmid b$ for any prime number $p \ge 7$;
        \item $5^3 \nmid b$, $3^3 \nmid b$;
        \item if $2^2 \mid b$, then $b \in \{4, 8 \}$.
    \end{itemize}
\end{lemma}
\begin{proof}
    The proof of this lemma can be carried out in a manner that is almost entirely analogous to the proof of Lemma \ref{weird_lemma_2}. To be more precise, the results $p^2 \nmid b$ for any prime $p \ge 7$, $5^3 \nmid b$ and $3^3 \nmid b$ can all be shown by simply implementing Lemma \ref{weird_lemma_3} together with the same idea used in the aforementioned proof of Lemma~\ref{weird_lemma_2}. Thus, we decide to leave out this part of the proof and focus solely on demonstrating $4 \mid b \implies b \in \{ 4, 8 \}$. We will do this separately for $U_t(x)$ and $W_t(x)$ by splitting the remaining piece of the problem into two corresponding cases.

    \bigskip\noindent
    \emph{Case $U_t(x)$}.\quad
    For a given even $t \ge 4$, suppose that $\Phi_b(x) \mid U_t(x)$ for some $b \in \mathbb{N}$ such that $4 \mid b$. Since the integers $2t-1, t+3, t+1, t-1, t-3$ are odd, while $t+2, t, t-2, t-4, 0$ are even, it is easy to use the same logic displayed in the proof of Lemma \ref{weird_lemma_2} in order to deduce that
    \begin{align*}
        \Phi_b(x) &\mid 2x^{2t-1} + x^{t+3}  + x^{t+1}  + 3x^{t-1} + x^{t-3},\\
        \Phi_b(x) &\mid - x^{t+2} - 3x^t - x^{t-2} - x^{t-4} - 2 .
    \end{align*}
    If we denote
    \begin{align*}
        A(x) &= 2x^{2t-1} + x^{t+3}  + x^{t+1}  + 3x^{t-1} + x^{t-3},\\
        B(x) &= - x^{t+2} - 3x^t - x^{t-2} - x^{t-4} - 2,
    \end{align*}
    we immediately obtain
    \begin{alignat*}{2}
        && \Phi_b(x) &\mid (x^6 + 3x^4 + x^2 + 1) \, A(x) + 2x^{t+3} \, B(x)\\
        \implies \quad && \Phi_b(x) &\mid x^{t+9} + 4x^{t+7} + 7x^{t+5} + 8x^{t+3} + 7x^{t+1} + 4x^{t-1} + x^{t-3}\\
        \implies \quad && \Phi_b(x) &\mid x^{t-3} \, (x^2+1)^4 (x^4 + 1)\\
        \implies \quad && \Phi_b(x) &\mid (x^2+1)^4 (x^4 + 1) .
    \end{alignat*}
    From here, it follows that any $b$-th primitive root of unity must also be a root of at least one of the two polynomials $x^2 + 1, x^4 + 1 \in \mathbb{Q}[x]$. Hence, $b \in \{4, 8 \}$.

    \bigskip\noindent
    \emph{Case $W_t(x)$}.\quad
    In this case, let $t \ge 4$ be some even integer and let $\Phi_b(x) \mid W_t(x)$ be true for some $b \in \mathbb{N}$ such that $4 \mid b$. In an absolutely analogous way as in the previous case, we conclude that
    \begin{align*}
        \Phi_b(x) &\mid 2x^{2t-1} - x^{t+3} - x^{t+1} + x^{t-1} - x^{t-3},\\
        \Phi_b(x) &\mid x^{t+2} - x^t + x^{t-2} + x^{t-4} - 2 .
    \end{align*}
    By denoting
    \begin{align*}
        A(x) &= 2x^{2t-1} - x^{t+3} - x^{t+1} + x^{t-1} - x^{t-3},\\
        B(x) &= x^{t+2} - x^t + x^{t-2} + x^{t-4} - 2,
    \end{align*}
    we swiftly get
    \begin{alignat*}{2}
        && \Phi_b(x) &\mid (-x^6 + x^4 - x^2 - 1) \, A(x) + 2x^{t+3} \, B(x)\\
        \implies \quad && \Phi_b(x) &\mid x^{t+9} - x^{t+5} - x^{t+1} + x^{t-3}\\
        \implies \quad && \Phi_b(x) &\mid x^{t-3} \, (x-1)^2 (x+1)^2 (x^2+1)^2 (x^4 + 1)\\
        \implies \quad && \Phi_b(x) &\mid (x^2+1)^2 (x^4 + 1) .
    \end{alignat*}
    Thus, once again we see that any $b$-th primitive root of unity must also be a root of at least one of the two polynomials $x^2 + 1, x^4 + 1 \in \mathbb{Q}[x]$, from which we quickly obtain $b \in \{4, 8\}$, as desired.
\end{proof}

Lemma \ref{weird_lemma_4} tells us that only certain cyclotomic polynomials could divide the $U_t(x)$ and $W_t(x)$ polynomials. In fact, except potentially the four polynomials $\Phi_1(x), \Phi_2(x), \Phi_4(x), \Phi_8(x)$, no other cyclotomic polynomial can divide $U_t(x)$ or $W_t(x)$, for each even $t \ge 4$. We shall now prove this claim by strongly relying on the next two auxiliary lemmas that greatly resemble the previously disclosed Lemmas~\ref{p_2p_1} and \ref{concrete_b_1}.

\begin{lemma}\label{p_2p_2}
    For each even $t \ge 4$ and each prime number $p \ge 11$, neither $U_t(x)$ nor $W_t(x)$ can be divisible by $\Phi_p(x)$ or $\Phi_{2p}(x)$.
\end{lemma}
\begin{proof}
    This lemma can be proved in an absolutely analogous manner as Lemma~\ref{p_2p_1}, the only difference being that Lemma \ref{unique_remainder_2} is used in place of Lemma \ref{unique_remainder_1}. For this reason, we choose to omit the proof details.
\end{proof}
\newpage
\begin{lemma}\label{concrete_b_2}
    For each even $t \ge 4$, neither $U_t(x)$ nor $W_t(x)$ can be divisible by a cyclotomic polynomial $\Phi_b(x)$ where $b \ge 3$ is a positive integer such that
    \begin{itemize}
        \item it does not have any prime factors outside of the set $\{2, 3, 5, 7\}$;
        \item it does not contain all the prime factors from the set $\{3, 5, 7 \}$;
        \item $2^2 \nmid b$, $3^3 \nmid b$, $5^3 \nmid b$, $7^2 \nmid b$.
    \end{itemize}
\end{lemma}
\begin{proof}
    This lemma can be proved in an absolutely analogous manner as Lemma \ref{concrete_b_1}, hence we choose the leave out the proof details. The corresponding computational results can be found in Appendix \ref{uwt_inspection}.
\end{proof}

We will now prove the previously mentioned statement regarding the divisibility of $U_t(x)$ and $W_t(x)$ polynomials by cyclotomic polynomials. In order to accomplish this, we shall strongly rely on Theorem \ref{filaseta} in a similar way as we have already done so while proving Lemmas \ref{qt_cyclotomic} and \ref{rt_cyclotomic}.

\begin{lemma}\label{uwt_cyclotomic}
    For each even $t \ge 4$ and each positive integer $b \in \mathbb{N}$ such that $b \notin \{1, 2, 4, 8\}$, neither $U_t(x)$ nor $W_t(x)$ are divisible by $\Phi_b(x)$.
\end{lemma}
\begin{proof}
    Suppose that $\Phi_b(x) \mid U_t(x)$ or $\Phi_b(x) \mid W_t(x)$ for some even $t \ge 4$ and some $b \notin \{1, 2, 4, 8 \}$. We will now finalize the proof via contradiction by splitting the problem into two separate cases depending on whether $b$ is divisible by a prime number from the set $\{3, 5, 7 \}$.

    \bigskip\noindent
    \emph{Case $3 \nmid b \land 5 \nmid b \land 7 \nmid b$}.\quad
    By implementing Lemma \ref{weird_lemma_4}, it becomes evident that $b$ has at least one prime factor greater than $7$, given the fact that $b \notin \{1, 2, 4, 8 \}$. Further on, we see that both $U_t(x)$ and $W_t(x)$ have at most $10$ non-zero terms, which means that Theorem \ref{filaseta} can be applied to any prime $p \ge 11$ in an analogous manner as it was done in the proof of Lemma \ref{qt_cyclotomic}. By cancelling out every single prime divisor of $b$ greater than $7$ until exactly one is left, we conclude that
    \[
        \Phi_{b'}(x) \mid Q_t(x) \quad \lor \quad \Phi_{b'}(x) \mid W_t(x)
    \]
    where $b'$ has a single prime divisor greater than $7$, and is potentially divisible by two as well, but not by four. By virtue of Lemma \ref{weird_lemma_4}, it is not difficult to deduce that $b'$ must either represent a prime number $p \ge 11$ or have the form $2p$ for some prime $p \ge 11$. Either way, Lemma \ref{p_2p_2} swiftly leads us to a contradiction.

    \bigskip\noindent
    \emph{Case $3 \mid b \lor 5 \mid b \lor 7 \mid b$}.\quad
    In this case, Theorem \ref{filaseta} can be applied in an analogous manner in order to cancel out any potential prime divisor of $b$ greater than $7$ until we obtain
    \[
        \Phi_{b'}(x) \mid Q_t(x) \quad \lor \quad \Phi_{b'}(x) \mid W_t(x)
    \]
    for some $b' \in \mathbb{N}$ such that all of its prime factors belong to the set $\{2, 3, 5, 7 \}$ and $3 \mid b' \lor 5 \mid b' \lor 7 \mid b'$. It is clear that $b' \ge 3$. By using Lemma \ref{weird_lemma_4}, we now see that $7^2 \nmid b'$, $5^3 \nmid b'$, $3^3 \nmid b'$, as well as $2^2 \nmid b'$. By virtue of Theorem \ref{filaseta}, we can suppose without loss of generality that $b'$ is not divisible by all the elements from the set $\{3, 5, 7\}$, due to the fact that $(3-2)+(5-2)+(7-2) > 10 - 2$. Taking everything into consideration, we conclude that $b'$ must satisfy the criteria given in Lemma~\ref{concrete_b_2}, which immediately yields a contradiction once more.
\end{proof}

We shall now implement Lemma \ref{uwt_cyclotomic} in order to finalize the proof of Theorem~\ref{main_th_2} in a similar manner as we have done so with Theorem \ref{main_th_1}.

\bigskip\noindent
\emph{Proof of Theorem \ref{main_th_1}}.\quad
Eq.\ (\ref{polynomial_formula}) directly gives us
\[
    P(\zeta) = \left( \zeta^{\frac{n}{4}-1} + \frac{1}{\zeta^{\frac{n}{4}-1}} \right) + \left( \zeta^{\frac{n}{4}+3} + \frac{1}{\zeta^{\frac{n}{4}+3}} \right) + \sum_{j=1}^{t-1}\left( \zeta^j + \frac{1}{\zeta^j}\right) + \sum_{j=\frac{n}{2}-t+1}^{\frac{n}{2}-1}\left( \zeta^j + \frac{1}{\zeta^j}\right),
\]
where $\zeta$ is an arbitrarily chosen $n$-th root of unity different from $1$ and $-1$. From here, we immediately get
\begin{equation}\label{aux_9}
    P(\zeta) = \left( \zeta^{\frac{n}{4}-1} + \frac{1}{\zeta^{\frac{n}{4}-1}} \right) + \left( \zeta^{\frac{n}{4}+3} + \frac{1}{\zeta^{\frac{n}{4}+3}} \right) + \sum_{j=1}^{t-1}\left( \zeta^j + \frac{1}{\zeta^j} + \zeta^{\frac{n}{2} - j} + \frac{1}{\zeta^{\frac{n}{2}-j}} \right) .
\end{equation}
We now divide the problem into two cases depending on whether $\zeta^\frac{n}{2}$ is equal to $1$ or $-1$. We shall finalize the proof by showing that $P(\zeta) \neq 0$ is certainly true in both cases.

\bigskip\noindent
\emph{Case $\zeta^\frac{n}{2} = -1$}.\quad
It is obvious that $\zeta^{\frac{n}{2} - j} = -\dfrac{1}{\zeta^j}$ and $\dfrac{1}{\zeta^{\frac{n}{2} - j}} = -\zeta^j$, from which we get
\[
    \zeta^j + \frac{1}{\zeta^j} + \zeta^{\frac{n}{2} - j} + \frac{1}{\zeta^{\frac{n}{2}-j}} = 0
\]
for any $j = \overline{1, t-1}$. For this reason, Eq.\ (\ref{aux_9}) simplifies to
\[
    P(\zeta) = \zeta^{\frac{n}{4}+3} + \zeta^{\frac{n}{4}-1} + \frac{1}{\zeta^{\frac{n}{4}-1}} + \frac{1}{\zeta^{\frac{n}{4}+3}} .
\]
The condition $P(\zeta) = 0$ now becomes equivalent to
\begin{alignat*}{2}
    && P(\zeta) &= 0\\
    \iff \quad && \zeta^{\frac{n}{4}+3} \left( \zeta^{\frac{n}{4}+3} + \zeta^{\frac{n}{4}-1} + \frac{1}{\zeta^{\frac{n}{4}-1}} + \frac{1}{\zeta^{\frac{n}{4}+3}} \right) &=0\\
    \iff \quad && \zeta^{\frac{n}{2} + 6} + \zeta^{\frac{n}{2} + 2} + \zeta^4 + 1 &= 0\\
    \iff \quad && -\zeta^6 - \zeta^2 + \zeta^4 + 1 &= 0\\
    \iff \quad && -(\zeta-1)(\zeta+1)(\zeta^4 + 1) &= 0\\
    \iff \quad && \zeta^4 &= -1 .
\end{alignat*}
Thus, $P(\zeta) = 0$ holds if and only if $\zeta$ is a primitive eighth root of unity. However, $8 \nmid n$, hence no $n$-th root of unity can be a primitive eighth root of unity. Thus, $P(\zeta) \neq 0$, as desired.

\bigskip\noindent
\emph{Case $\zeta^\frac{n}{2} = 1$}.\quad
In this case, it is clear that $\zeta^{\frac{n}{2} - j} = \dfrac{1}{\zeta^j}$ and $\dfrac{1}{\zeta^{\frac{n}{2} - j}} = \zeta^j$, which immediately leads us to
\[
    \zeta^j + \frac{1}{\zeta^j} + \zeta^{\frac{n}{2} - j} + \frac{1}{\zeta^{\frac{n}{2}-j}} = 2 \left( \zeta^j + \frac{1}{\zeta^j} \right)
\]
for any $j = \overline{1, t-1}$. Bearing this in mind, it is straightforward to see that
\allowdisplaybreaks
\begin{alignat*}{2}
    && P(\zeta) &= 0\\
    \iff \quad && \left( \zeta^{\frac{n}{4}+3} + \zeta^{\frac{n}{4}-1} + \frac{1}{\zeta^{\frac{n}{4}-1}} + \frac{1}{\zeta^{\frac{n}{4}+3}} \right) + 2 \sum_{j=1}^{t-1}\left( \zeta^j + \frac{1}{\zeta^j} \right) &= 0\\
    \iff \quad && \left( \zeta^{\frac{n}{4}+3} + \zeta^{\frac{n}{4}-1} + \frac{1}{\zeta^{\frac{n}{4}-1}} + \frac{1}{\zeta^{\frac{n}{4}+3}} \right) - 2 + 2 \sum_{j=-t+1}^{t-1} \zeta^j &= 0\\
    \iff \quad && \zeta^{t-1} \left( \zeta^{\frac{n}{4}+3} + \zeta^{\frac{n}{4}-1} + \frac{1}{\zeta^{\frac{n}{4}-1}} + \frac{1}{\zeta^{\frac{n}{4}+3}} - 2 + 2 \sum_{j=-t+1}^{t-1} \zeta^j \right) &= 0\\
    \iff \quad && \zeta^{t+\frac{n}{4}+2} + \zeta^{t+\frac{n}{4}-2} + \zeta^{t-\frac{n}{4}} + \zeta^{t-\frac{n}{4}-4} - 2\zeta^{t-1} + 2 \sum_{j=0}^{2t-2} \zeta^j &= 0\\
    \iff \quad && (\zeta - 1) \left( \zeta^{t+\frac{n}{4}+2} + \zeta^{t+\frac{n}{4}-2} + \zeta^{t-\frac{n}{4}} + \zeta^{t-\frac{n}{4}-4} - 2\zeta^{t-1} + 2 \sum_{j=0}^{2t-2} \zeta^j \right) &= 0 ,
\end{alignat*}
which finally means that $P(\zeta) = 0$ must be equivalent to
\begin{align}\label{aux_10}
    \begin{split}
    &\zeta^{t + \frac{n}{4} + 3} - \zeta^{t + \frac{n}{4} + 2} + \zeta^{t + \frac{n}{4} - 1} - \zeta^{t + \frac{n}{4} - 2} + \zeta^{t - \frac{n}{4} + 1} - \zeta^{t - \frac{n}{4}} + \zeta^{t - \frac{n}{4} - 3} - \zeta^{t - \frac{n}{4} - 4}\\
    &\qquad + 2 \zeta^{2t-1} - 2\zeta^t + 2\zeta^{t-1} - 2 = 0 .
    \end{split}
\end{align}
We now split the problem into two separate subcases depending on whether $\zeta^\frac{n}{4}$ is equal to $1$ or $-1$.

\interdisplaylinepenalty=10000
\medskip\noindent
\emph{Subcase $\zeta^{\frac{n}{4}} = 1$}.\quad
In this subcase, the implementation of Eq.\ (\ref{aux_10}) directly gives that $P(\zeta) = 0$ is further equivalent to
\begin{align*}
    &\zeta^{t + 3} - \zeta^{t + 2} + \zeta^{t - 1} - \zeta^{t - 2} + \zeta^{t + 1} - \zeta^{t} + \zeta^{t - 3} - \zeta^{t - 4}\\
    &\qquad + 2 \zeta^{2t-1} - 2\zeta^t + 2\zeta^{t-1} - 2 = 0 ,
\end{align*}
that is
\begin{align*}
    2 \zeta^{2t-1} + \zeta^{t+3} - \zeta^{t + 2} + \zeta^{t + 1} - 3\zeta^{t} + 3 \zeta^{t - 1} - \zeta^{t - 2} + \zeta^{t - 3} - \zeta^{t - 4} - 2 &= 0 .
\end{align*}
In other words, $P(\zeta) = 0$ is equivalent to $\zeta$ being a root of $U_t(x)$. Now, we know that $\zeta \neq 1, -1$ and that $\zeta$ cannot be a primitive eighth root of unity, as discussed earlier. On top of that, $\zeta \neq i, -i$ given the fact that $\frac{n}{2} \equiv_4 2$, hence $i^\frac{n}{2} = (-i)^\frac{n}{2} = -1$. Bearing this in mind, it is clear that if were $P(\zeta) = 0$ were to hold, then $U_t(x)$ would be divisible by some cyclotomic polynomial $\Phi_b(x)$ where $b \notin \{1, 2, 4, 8 \}$. However, this is not possible according to Lemma \ref{uwt_cyclotomic}. For this reason, $P(\zeta) \neq 0$ must hold.

\medskip\noindent
\emph{Subcase $\zeta^{\frac{n}{4}} = -1$}.\quad
In this scenario, Eq.\ (\ref{aux_10}) can be quickly simplified to
\begin{align*}
    &-\zeta^{t + 3} + \zeta^{t + 2} - \zeta^{t - 1} + \zeta^{t - 2} - \zeta^{t + 1} + \zeta^{t} - \zeta^{t - 3} + \zeta^{t - 4}\\
    &\qquad + 2 \zeta^{2t-1} - 2\zeta^t + 2\zeta^{t-1} - 2 = 0 ,
\end{align*}
that is
\begin{align*}
    2 \zeta^{2t-1} - \zeta^{t+3} + \zeta^{t + 2} - \zeta^{t + 1} - \zeta^{t} + \zeta^{t - 1} + \zeta^{t - 2} - \zeta^{t - 3} + \zeta^{t - 4} - 2 &= 0 .
\end{align*}
Thus, we get that $P(\zeta) = 0$ is equivalent to $\zeta$ being a root of $W_t(x)$. By using the analogous logic as in the previous subcase, it is easy to establish that $P(\zeta) = 0$ would imply that $W_t(x)$ is divisible by some cyclotomic polynomial $\Phi_b(x)$ where $b \notin \{1, 2, 4, 8\}$, which is again impossible due to Lemma \ref{uwt_cyclotomic}. Hence, $P(\zeta) = 0$ cannot be true, which completes the proof. \hfill\qed

\section{Conclusion}\label{conclusion}

Theorem \ref{main_theorem} provides the full answer to the question posed by the circulant nut graph order--degree existence problem. It is evident that there exists a clear and rich pattern that the orders and degrees of circulant nut graphs must follow, with the sole exception being the case $(n, d) = (16, 8)$. Bearing this in mind, it now becomes possible to explore other types of nut graphs more easily.

For example, it is clear that each circulant graph is necessarily a Cayley graph, which is, in turn, surely a vertex-transitive graph. For this reason, if we are trying to investigate the existence of Cayley nut graphs or vertex-transitive nut graphs, Theorem \ref{main_theorem} provides a solid starting point. Taking all the aforementioned facts into consideration, we are able to disclose the following corollary.

\begin{corollary}
    Let $d \in \mathbb{N}$ be such that $4 \mid d$, and let $n \in \mathbb{N}$ be such that $2 \mid n$ and
    \begin{itemize}
        \item $n \ge d + 4$ if $8 \nmid d$;
        \item $n \ge d + 6$ if $8 \mid d$;
        \item $(n, d) \neq (16, 8)$.
    \end{itemize}
    There necessarily exists a $d$-regular Cayley nut graph of order $n$, as well as a $d$-regular vertex-transitive nut graph of order $n$.
\end{corollary}

One of the possible ways of extending the research of nut graphs is by considering the order--degree existence problem regarding Cayley nut graphs or vertex-transitive nut graphs, whose conditions are less restrictive than those corresponding to the circulant nut graphs. Bearing this in mind, we end the paper with the following two open problems.

\begin{problem}
    What are all the pairs $(n, d)$ for which there exists a $d$-regular Cayley nut graph of order $n$?
\end{problem}
\begin{problem}
    What are all the pairs $(n, d)$ for which there exists a $d$-regular vertex-transitive nut graph of order $n$?
\end{problem}

\appendix

\section{Inspection for \texorpdfstring{$\Phi_b(x) \nmid Q_t(x)$}{Phib(x) not | Qt(x)}}
\label{qt_inspection}

In this appendix section, we will disclose the computational results that demonstrate $\Phi_b(x) \nmid Q_t(x)$ for all the required values of $b \ge 3$, as stated in Lemma~\ref{concrete_b_1}. First of all, the set of all such values of $b$ can be obtained in a plethora of ways. For example, the following short Python script can be used.
\begin{lstlisting}[language = Python, frame = trBL, escapeinside={(*@}{@*)}, aboveskip=10pt, belowskip=10pt, numbers=left, rulecolor=\color{black}]
import numpy as np


def main():
    part_1 = np.multiply.outer([1, 2], [1, 3, 9, 27]).reshape(-1)
    part_2 = np.multiply.outer([1, 5, 25], [1, 7, 49]).reshape(-1)

    all_of_them = np.multiply.outer(part_1, part_2).reshape(-1)
    all_of_them.sort()
    all_of_them = all_of_them.tolist()

    result = list(filter(lambda item: item % 105 != 0, all_of_them))
    result = list(filter(lambda item: item >= 3, result))

    print(len(result))
    print(result)


if __name__ == "__main__":
    main()
\end{lstlisting}
The said script quickly finds that there exist exactly $46$ values of $b$ that satisfy the given criteria:
\begin{align*}
    \{ &3, 5, 6, 7, 9, 10, 14, 15, 18, 21, 25, 27, 30, 35, 42, 45, 49, 50, 54, 63,\\
    &70, 75, 90, 98, 126, 135, 147, 150, 175, 189, 225, 245, 270, 294,\\
    &350, 378, 441, 450, 490, 675, 882, 1225, 1323, 1350, 2450, 2646 \} .
\end{align*}
For each of these values, it can be determined that $\Phi_b(x) \nmid Q_t^{\bmod b}(x)$ for any possible remainder $t \bmod b$. In order to achieve this, we can use, for example, the following Mathematica command.
\newpage
\begin{lstlisting}[language = Mathematica, frame = trBL, escapeinside={(*@}{@*)}, aboveskip=10pt, belowskip=10pt, numbers=left, rulecolor=\color{black}, morekeywords={CoefficientRules}]
Min[Table[
  Min[Table[
    Length[CoefficientRules[
      PolynomialRemainder[
       2 x^Mod[2 t + 1, b] - 2 x^Mod[2 t - 1, b] + 
        2 x^Mod[2 t - 2, b] + x^Mod[t + 3, b] - x^Mod[t + 2, b] + 
        x^Mod[t - 1, b] - x^Mod[t - 2, b] - 2 x^3 + 2 x^2 - 2, 
       Cyclotomic[b, x], x]]], {t, 0, b - 1}]], {b, {3, 5, 6, 7, 9, 
    10, 14, 15, 18, 21, 25, 27, 30, 35, 42, 45, 49, 50, 54, 63, 70, 
    75, 90, 98, 126, 135, 147, 150, 175, 189, 225, 245, 270, 294, 350,
     378, 441, 450, 490, 675, 882, 1225, 1323, 1350, 2450, 2646}}]]
\end{lstlisting}
This command yields the minimum possible number of non-zero terms that the polynomial $Q_t^{\bmod b}(x) \bmod \Phi_b(x)$ can have, as $b$ ranges through all the necessary values and $t \bmod b$ varies through all the possible remainders. It can be promptly checked that this number is equal to one, which immediately means that the remainder $Q_t^{\bmod b}(x) \bmod \Phi_b(x)$ is never equal to the zero polynomial. From here, we quickly obtain our desired result.

\section{Inspection for \texorpdfstring{$\Phi_b(x) \nmid R_t(x)$}{Phib(x) not | Rt(x)}}
\label{rt_inspection}

Here, we elaborate the computational results showing that $\Phi_b(x) \nmid R_t(x)$ for all the values of $b \ge 3$ given in Lemma~\ref{concrete_b_1}. For starters, the set of all such values of $b$ can be computed, for example, by using the following Python script.
\begin{lstlisting}[language = Python, frame = trBL, escapeinside={(*@}{@*)}, aboveskip=10pt, belowskip=10pt, numbers=left, rulecolor=\color{black}]
import numpy as np


def main():
    part_1 = np.multiply.outer([1, 2], [1, 7, 11, 77]).reshape(-1)
    part_2 = np.multiply.outer([1, 3, 9], [1, 5, 25]).reshape(-1)

    all_of_them = np.multiply.outer(part_1, part_2).reshape(-1)
    all_of_them.sort()
    all_of_them = all_of_them.tolist()

    result = list(filter(lambda item: item % 77 != 0, all_of_them))
    result = list(filter(lambda item: item % 55 != 0, result))
    result = list(filter(lambda item: item >= 3, result))

    print(len(result))
    print(result)


if __name__ == "__main__":
    main()
\end{lstlisting}
The Python script easily concludes that there exist exactly $40$ values of $b$ that satisfy the given criteria:
\begin{align*}
    \{ &3, 5, 6, 7, 9, 10, 11, 14, 15, 18, 21, 22, 25, 30, 33, 35, 42,\\
    &45, 50, 63, 66, 70, 75, 90, 99, 105, 126, 150, 175, 198,\\
    &210, 225, 315, 350, 450, 525, 630, 1050, 1575, 3150 \} .
\end{align*}
For each of these values, it can be determined that $\Phi_b(x) \nmid R_t^{\bmod b}(x)$ regardless of what the remainder $t \bmod b$ is. This can be done, for example, by using the next Mathematica command.
\begin{lstlisting}[language = Mathematica, frame = trBL, escapeinside={(*@}{@*)}, aboveskip=10pt, belowskip=10pt, numbers=left, rulecolor=\color{black}, morekeywords={CoefficientRules}]
Min[Table[
  Min[Table[
    Length[CoefficientRules[
      PolynomialRemainder[
       2 x^Mod[2 t + 1, b] - 2 x^Mod[2 t - 1, b] + 
        2 x^Mod[2 t - 2, b] - x^Mod[t + 3, b] + x^Mod[t + 2, b] - 
        4 x^Mod[t + 1, b] + 4 x^Mod[t, b] - x^Mod[t - 1, b] + 
        x^Mod[t - 2, b] - 2 x^3 + 2 x^2 - 2, Cyclotomic[b, x], 
       x]]], {t, 0, b - 1}]], {b, {3, 5, 6, 7, 9, 10, 11, 14, 15, 18, 
    21, 22, 25, 30, 33, 35, 42, 45, 50, 63, 66, 70, 75, 90, 99, 105, 
    126, 150, 175, 198, 210, 225, 315, 350, 450, 525, 630, 1050, 1575,
     3150}}]]
\end{lstlisting}
The said command computes the minimum possible number of non-zero terms that the polynomial $R_t^{\bmod b}(x) \bmod \Phi_b(x)$ can have, as $b$ ranges through all the required values and $t \bmod b$ varies through all the possible remainders. The computation output is equal to one, hence we obtain our desired result in the same way as in Appendix \ref{qt_inspection}.

\section{Inspection for \texorpdfstring{$\Phi_b(x) \nmid U_t(x)$}{Phib(x) not | Ut(x)} and \texorpdfstring{$\Phi_b(x) \nmid W_t(x)$}{Phib(x) not | Wt(x)}}
\label{uwt_inspection}

We can use an analogous mechanism to disclose the computational results that demonstrate $\Phi_b(x) \nmid U_t(x)$, as well as $\Phi_b(x) \nmid W_t(x)$, for all the values of $b \ge 3$ stated in Lemma \ref{concrete_b_2}. The set of all the required values of $b$ can be determined, for example, by using the following Python script.
\begin{lstlisting}[language = Python, frame = trBL, escapeinside={(*@}{@*)}, aboveskip=10pt, belowskip=10pt, numbers=left, rulecolor=\color{black}]
import numpy as np


def main():
    part_1 = np.multiply.outer([1, 2], [1, 3, 9]).reshape(-1)
    part_2 = np.multiply.outer([1, 5, 25], [1, 7]).reshape(-1)

    all_of_them = np.multiply.outer(part_1, part_2).reshape(-1)
    all_of_them.sort()
    all_of_them = all_of_them.tolist()

    result = list(filter(lambda item: item % 105 != 0, all_of_them))
    result = list(filter(lambda item: item >= 3, result))

    print(len(result))
    print(result)


if __name__ == "__main__":
    main()
\end{lstlisting}
The given Python script finds that there exist exactly $26$ values of $b$ that satisfy the given criteria:
\begin{align*}
    \{ &3, 5, 6, 7, 9, 10, 14, 15, 18, 21, 25, 30, 35, 42, 45,\\
    &50, 63, 70, 75, 90, 126, 150, 175, 225, 350, 450 \} .
\end{align*}
For each of these values, it can be promptly shown that $\Phi_b(x) \nmid U_t^{\bmod b}(x)$ and $\Phi_b(x) \nmid W_t^{\bmod b}(x)$ for any possible remainder $t \bmod b$. This can be accomplished, for example, by using the following two Mathematica commands.
\begin{lstlisting}[language = Mathematica, frame = trBL, escapeinside={(*@}{@*)}, aboveskip=10pt, belowskip=10pt, numbers=left, rulecolor=\color{black}, morekeywords={CoefficientRules}]
Min[Table[
  Min[Table[
    Length[CoefficientRules[
      PolynomialRemainder[
       2 x^Mod[2 t - 1, b] + x^Mod[t + 3, b] - x^Mod[t + 2, b] + 
        x^Mod[t + 1, b] - 3 x^Mod[t, b] + 3 x^Mod[t - 1, b] - 
        x^Mod[t - 2, b] + x^Mod[t - 3, b] - x^Mod[t - 4, b] - 2, 
       Cyclotomic[b, x], x]]], {t, 0, b - 1}]], {b, {3, 5, 6, 7, 9, 
    10, 14, 15, 18, 21, 25, 30, 35, 42, 45, 50, 63, 70, 75, 90, 126, 
    150, 175, 225, 350, 450}}]]
\end{lstlisting}
\begin{lstlisting}[language = Mathematica, frame = trBL, escapeinside={(*@}{@*)}, aboveskip=10pt, belowskip=10pt, numbers=left, rulecolor=\color{black}, morekeywords={CoefficientRules}]
Min[Table[
  Min[Table[
    Length[CoefficientRules[
      PolynomialRemainder[
       2 x^Mod[2 t - 1, b] - x^Mod[t + 3, b] + x^Mod[t + 2, b] - 
        x^Mod[t + 1, b] - x^Mod[t, b] + x^Mod[t - 1, b] + 
        x^Mod[t - 2, b] - x^Mod[t - 3, b] + x^Mod[t - 4, b] - 2, 
       Cyclotomic[b, x], x]]], {t, 0, b - 1}]], {b, {3, 5, 6, 7, 9, 
    10, 14, 15, 18, 21, 25, 30, 35, 42, 45, 50, 63, 70, 75, 90, 126, 
    150, 175, 225, 350, 450}}]]
\end{lstlisting}
The given two commands yield the minimum possible number of non-zero terms that the polynomials $U_t^{\bmod b}(x)$ and $W_t^{\bmod b}(x)$ can have, respectively, as $b$ ranges through all the required values and $t \bmod b$ takes on any possible value. Both computation outputs are equal to one, which means that none of the aforementioned remainders are equal to the zero polynomial, as desired.

\section{Roots of certain polynomials}\label{problematic_roots}

In this appendix section, we will demonstrate that none of the following polynomials
\begin{align*}
    Z_1(x) &= x^4 + 2x^3 - 2x^2 + 2x + 1 ,\\
    Z_2(x) &= x^6 - x^4 + 2x^3 - x^2 + 1 ,\\
    Z_3(x) &= x^6 - 2x^5 + 3x^4 - 2x^3 + 3x^2 - 2x + 1 ,\\
    Z_4(x) &= 3x^4 - 2x^2 + 3,\\
    Z_5(x) &= x^2 - 2x - 1,\\
    Z_6(x) &= x^2 + 2x - 1,\\
    Z_7(x) &= 3x^4 + 2x^2 + 3,\\
    Z_8(x) &= 2x^6 - 2x^5 + 3x^4 - 2x^3 + 3x^2 - 2x + 2,\\
    Z_9(x) &= 2x^8 + 2x^7 - x^6 + 2x^5 - 2x^4 + 2x^3 - x^2 + 2x + 2 ,
\end{align*}
contain a root of unity among its roots. This can be swiftly achieved by simply showing that none of them are divisible by any cyclotomic polynomial $\Phi_b(x)$. In fact, it is clear that, for each $j = \overline{1, 9}$, the polynomial $Z_j(x)$ cannot be divisible by a $\Phi_b(x)$ such that $\deg\Phi_b > \deg Z_j$. Thus, in order to prove the desired result, it is sufficient to show that each given polynomial is not divisible by the corresponding cyclotomic polynomials whose degrees do not exceed its own. This is trivial to accomplish via computer.

For starters, it is not difficult to determine all $18$ cyclotomic polynomials whose degree is not above $8$:
\begin{align*}
    \Phi_1(x) &= x - 1, & \Phi_7(x) &= x^6 + x^5 + x^4 + x^3 + x^2 + x + 1,\\
    \Phi_2(x) &= x + 1, & \Phi_9(x) &= x^6 + x^3 + 1,\\
    \Phi_3(x) &= x^2 + x + 1, & \Phi_{14}(x) &= x^6 - x^5 + x^4 - x^3 + x^2 - x + 1,\\
    \Phi_4(x) &= x^2 + 1, & \Phi_{18}(x) &= x^6 - x^3 + 1,\\
    \Phi_6(x) &= x^2 - x + 1, & \Phi_{15}(x) &= x^8 - x^7 + x^5 - x^4 + x^3 - x + 1,\\
    \Phi_5(x) &= x^4 + x^3 + x^2 + x + 1, & \Phi_{16}(x) &= x^8 + 1,\\
    \Phi_8(x) &= x^4 + 1, & \Phi_{20}(x) &= x^8 - x^6 + x^4 - x^2 + 1,\\
    \Phi_{10}(x) &= x^4 - x^3 + x^2 - x + 1, & \Phi_{24}(x) &= x^8 - x^4 + 1,\\
    \Phi_{12}(x) &= x^4 - x^2 + 1, & \Phi_{30}(x) &= x^8 + x^7 - x^5 - x^4 - x^3 + x + 1.
\end{align*}
The necessary computational results can be found on Tables \ref{problematic_roots_1}, \ref{problematic_roots_2}, \ref{problematic_roots_3} and \ref{problematic_roots_4}. The disclosed remainders clearly indicate that no given polynomial can be divisible by any cyclotomic polynomial of interest, as desired.

\begin{table}
\begin{center}
{\scriptsize
\begin{tabular}{lll}
\toprule $b$ & $Z_5(x) \bmod \Phi_b(x)$ & $Z_6(x) \bmod \Phi_b(x)$ \\
\midrule
$1$ & $-2$ & $2$ \\
$2$ & $2$ & $-2$ \\
$3$ & $-2-3 x$ & $-2+x$ \\
$4$ & $-2-2 x$ & $-2+2 x$ \\
$6$ & $-2-x$ & $-2+3 x$ \\
\bottomrule
\end{tabular}
}
\caption{The required remainders of $Z_5(x)$ and $Z_6(x)$.}      
\label{problematic_roots_1}
\end{center}
\end{table}

\begin{table}
\begin{center}
{\scriptsize
\begin{tabular}{llll}
\toprule $b$ & $Z_1(x) \bmod \Phi_b(x)$ & $Z_4(x) \bmod \Phi_b(x)$ & $Z_7(x) \bmod \Phi_b(x)$\\
\midrule
$1$ & $4$ & $4$ & $8$ \\
$2$ & $-4$ & $4$ & $8$ \\
$3$ & $5+5 x$ & $5+5 x$ & $1+x$ \\
$4$ & $4$ & $8$ & $4$ \\
$6$ & $1-x$ & $5-5 x$ & $1-x$ \\
$5$ & $x-3 x^2+x^3$ & $-3 x-5 x^2-3 x^3$ & $-3 x-x^2-3 x^3$ \\
$8$ & $2 x-2 x^2+2 x^3$ & $-2 x^2$ & $2 x^2$ \\
$10$ & $3 x-3 x^2+3 x^3$ & $3 x-5 x^2+3 x^3$ & $3 x-x^2+3 x^3$ \\
$12$ & $2 x-x^2+2 x^3$ & $x^2$ & $5 x^2$ \\
\bottomrule
\end{tabular}
}
\caption{The required remainders of $Z_1(x)$, $Z_4(x)$ and $Z_7(x)$.}      
\label{problematic_roots_2}
\end{center}
\end{table}

\begin{table}
\begin{center}
{\scriptsize
\begin{tabular}{llll}
\toprule $b$ & $Z_2(x) \bmod \Phi_b(x)$ & $Z_3(x) \bmod \Phi_b(x)$ & $Z_8(x) \bmod \Phi_b(x)$\\
\midrule
$1$ & $2$ & $2$ & $4$ \\
$2$ & $-2$ & $14$ & $16$ \\
$3$ & $5$ & $-1$ & $1$ \\
$4$ & $-2 x$ & $-2 x$ & $-2 x$ \\
$6$ & $1$ & $-1$ & $1$ \\
$5$ & $2+2 x+3 x^3$ & $-4-4 x-5 x^3$ & $-3-3 x-5 x^3$ \\
$8$ & $2-2 x^2+2 x^3$ & $-2+2 x^2-2 x^3$ & $-1+x^2-2 x^3$ \\
$10$ & $2-2 x+x^3$ & $x^3$ & $1-x+x^3$ \\
$12$ & $1-2 x^2+2 x^3$ & $-3+6 x^2-4 x^3$ & $-3+6 x^2-4 x^3$ \\
$7$ & $-x-2 x^2+x^3-2 x^4-x^5$ & $-3 x+2 x^2-3 x^3+2 x^4-3 x^5$ & $-4 x+x^2-4 x^3+x^4-4 x^5$ \\
$9$ & $-x^2+x^3-x^4$ & $-2 x+3 x^2-3 x^3+3 x^4-2 x^5$ & $-2 x+3 x^2-4 x^3+3 x^4-2 x^5$ \\
$14$ & $x-2 x^2+3 x^3-2 x^4+x^5$ & $-x+2 x^2-x^3+2 x^4-x^5$ & $x^2+x^4$ \\
$18$ & $-x^2+3 x^3-x^4$ & $-2 x+3 x^2-x^3+3 x^4-2 x^5$ & $-2 x+3 x^2+3 x^4-2 x^5$ \\
\bottomrule
\end{tabular}
}
\caption{The required remainders of $Z_2(x)$, $Z_3(x)$ and $Z_8(x)$.}      
\label{problematic_roots_3}
\end{center}
\end{table}

\begin{table}
\begin{center}
{\scriptsize
\begin{tabular}{ll}
\toprule $b$ & $Z_9(x) \bmod \Phi_b(x)$\\
\midrule
$1$ & $8$ \\
$2$ & $-8$ \\
$3$ & $-x$ \\
$4$ & $4$ \\
$6$ & $5 x$ \\
$5$ & $6+3 x+3 x^2+6 x^3$ \\
$8$ & $6$ \\
$10$ & $2+x-x^2-2 x^3$ \\
$12$ & $5-2 x-5 x^2+4 x^3$ \\
$7$ & $5+5 x+3 x^3-x^4+3 x^5$ \\
$9$ & $3-3 x^2+3 x^3-4 x^4$ \\
$14$ & $1-x+x^3-x^4+x^5$ \\
$18$ & $3-3 x^2+x^3+4 x^5$ \\
$15$ & $4 x-x^2-x^6+4 x^7$ \\
$16$ & $2 x-x^2+2 x^3-2 x^4+2 x^5-x^6+2 x^7$ \\
$20$ & $2 x+x^2+2 x^3-4 x^4+2 x^5+x^6+2 x^7$ \\
$24$ & $2 x-x^2+2 x^3+2 x^5-x^6+2 x^7$ \\
$30$ & $-x^2+4 x^3+4 x^5-x^6$ \\
\bottomrule
\end{tabular}
}
\caption{The required remainders of $Z_9(x)$.}      
\label{problematic_roots_4}
\end{center}
\end{table}

\end{document}